\documentclass[11pt,letterpaper,reqno]{amsart}
\usepackage{fullpage}
\usepackage[top=2cm, bottom=4.5cm, left=2.5cm, right=2.5cm]{geometry}
\usepackage{amsmath,amsthm,amsfonts,amssymb,amscd}
\usepackage{geometry,tikz,tikz-cd}
\usepackage{empheq}
\usepackage{lipsum}
\usepackage{lastpage}
\usepackage{thmtools}
\usepackage{enumerate}
\usepackage[shortlabels]{enumitem}
\usepackage{fancyhdr}
\usepackage{mathrsfs}
\usepackage{xcolor}
\usepackage{graphicx}
\usepackage{listings}
\usepackage{bbm}
\usepackage{array}
\usepackage{hyperref}
\usepackage[makeroom]{cancel}

\usepackage[utf8]{inputenc}
\usepackage{comment}

\usepackage{makecell}

\usepackage[all]{xy}
\xyoption{matrix}
\xyoption{arrow}

\usetikzlibrary{arrows,decorations.pathmorphing,backgrounds,positioning,fit,petri}

\tikzset{help lines/.style={step=#1cm,very thin, color=gray},
help lines/.default=.5} 
\tikzset{thick grid/.style={step=#1cm,thick, color=gray},
thick grid/.default=1}

\mathtoolsset{showonlyrefs=true}
\numberwithin{figure}{section}
\numberwithin{table}{section}

\theoremstyle{definition}
\newtheorem*{remark}{Remark}

\theoremstyle{plain}
\newcommand{\thistheoremname}{}
\newtheorem*{genericthm*}{\thistheoremname}
\newenvironment{namedthm*}[1]
  {\renewcommand{\thistheoremname}{#1}%
   \begin{genericthm*}}
  {\end{genericthm*}}

\hypersetup{%
  colorlinks=true,
  linkcolor=blue,
  citecolor=blue,
  linkbordercolor={0 0 1}
}

\DeclareMathOperator{\ZZ}{\mathbb{Z}}

\allowdisplaybreaks

\newtheorem{theorem}{Theorem}[section]

\newtheorem{lemma}[theorem]{Lemma}

\newtheorem{proposition}[theorem]{Proposition}
\theoremstyle{definition}
\newtheorem{definition}{Definition}
\theoremstyle{definition}
\newtheorem{example}[theorem]{Example}
\theoremstyle{remark}
\newtheorem{conjecture}[theorem]{\bf Conjecture}

\numberwithin{equation}{section}

\usepackage{latexsym}
\usepackage{amssymb}

\newcommand{\ben}{\begin{equation}}
\newcommand{\een}{\end{equation}}

\newcommand{\Ud}{U_{|D_2|}}
\newcommand{\Gk}{G_{k, D_1, D_2}}
\newcommand{\sig}{\sigma_{k-1,D_1, D_2}}

\newcommand{\SL}[1]{\ensuremath{{\mathrm {SL}_{ #1 }}}}

\newcommand{\legendre}[2]{\left(\hspace{-1pt}\frac{#1}{#2}\hspace{-1pt}\right)}

\makeatletter
\NewDocumentCommand{\sump}{e{_}}
 {%
  \DOTSB
  \mathop{\IfNoValueTF{#1}{\sump@{}}{\sump@{#1}}}%
  \nolimits
 }
\newcommand{\sump@}[1]{\mathpalette\sump@@{#1}}
\newcommand{\sump@@}[2]{%
  \ifx#1\displaystyle
    {\sump@display{#2}}%
  \else
    \sum@\nolimits'_{#2}%
  \fi
}
\newcommand{\sump@display}[1]{%
  \sbox\z@{$\m@th\displaystyle\sum@\nolimits'$}%
  \sbox\tw@{$\m@th\displaystyle\sum@\limits_{#1}$}%
  \sbox\@tempboxa{$\m@th\displaystyle'$}
  \mathop{\sum@\nolimits' \kern-\wd\@tempboxa}\limits_{#1}%
  \ifdim\wd\z@>\wd\tw@
    \kern\dimexpr\wd\z@-\wd\tw@\relax
  \fi
}
\makeatother

\newcommand{\genlegendre}[4]{%
  \genfrac{(}{)}{}{#1}{#3}{#4}%
  \if\relax\detokenize{#2}\relax\else_{\!#2}\fi
}

\newcommand{\BZ}{{\mathbb{Z}}}
\newcommand{\BP}{{\mathbb{P}}}

\newcommand{\BR}{{\mathbb{R}}}

\newcommand{\N}{{\bf {\rm N}}}

\DeclareMathOperator{\Sl}{SL}
\DeclareMathOperator{\Gl}{GL}
\DeclareMathOperator{\Span}{Span}

\DeclareMathOperator{\im}{Im}
\DeclareMathOperator{\re}{Re}

\DeclareMathOperator{\Tr}{Tr}

\DeclareMathOperator{\pr}{pr}

\setlist[enumerate]{leftmargin=*,widest=0}
\setlist[itemize]{leftmargin=*,widest=0}
\makeatletter
\def\subsection{\@startsection{subsection}{2}%
  \z@{.5\linespacing\@plus.7\linespacing}{.3\linespacing}%
  {\normalfont\bfseries}}

\def\subsubsection{\@startsection{subsubsection}{3}%
  \z@{.5\linespacing\@plus.7\linespacing}{.3\linespacing}%
  {\normalfont\bfseries}}
\makeatother

\title{Subspaces spanned by eigenforms with nonvanishing twisted central $L$-values}

\author[J. Kayath]{June Kayath}
\address[J. Kayath]{Department of Mathematics, Massachusetts Institute of Technology, Cambridge, MA, 02139, USA}
\email{kayath@mit.edu}

\author[C. Lane]{Connor Lane}
\address[C. Lane]{Deparment of Mathematics, Rose-Hulman Institute of Technology, Terre-Haute, IN, 47803, USA}
\email{lanecf@rose-hulman.edu}

\author[B. Neifeld]{Ben Neifeld}
\address[B. Neifeld]{Deparment of Mathematics, William {\&} Mary, Williamsburg, VA, 23187, USA}
\email{bmneifeld@wm.edu}

\author[T. Ni]{Tianyu Ni}
\address[T. Ni]{School of Mathematical and Statistical Sciences, Clemson University, Clemson, SC, 29634, USA}
\email{tianyuni1994math@gmail.com}

\author[H. Xue]{Hui Xue}
\address[H. Xue]{School of Mathematical and Statistical Sciences, Clemson University, Clemson, SC, 29634, USA}
\email{huixue@clemson.edu}


\begin{document}

\begin{abstract}
In this paper, we construct explicit spanning sets for two spaces. One is the subspace generated by integral-weight Hecke eigenforms with nonvanishing quadratic twisted central $L$-values. The other is a subspace generated by half-integral weight Hecke eigenforms with certain nonvanishing Fourier coefficients. Along the way, we show that these subspaces are isomorphic via the Shimura lift.

\end{abstract}

\maketitle

\section{Introduction}

Let $\ell\ge 2$ be an integer. For $N\geq1$ and a Dirichlet character $\chi$ modulo $N$, let $M_{\ell}(N,\chi)$ and $S_{\ell}(N,\chi)$ be the space of modular forms and cuspforms of weight $\ell$, level $N$ and nebentypus $\chi$, respectively. When $\chi$ is trivial, we simply write $M_{\ell}(N)$ and $S_{\ell}(N)$. Let $M_{\ell+1/2}(4N)$ and $S_{\ell+1/2}(4N)$ be the space of modular forms and the space of cuspforms of weight $\ell+1/2$ for $\Gamma_0(4N)$, respectively. For $N=1$ we recall the Kohnen \cite{Kohenhalfintegralweight1980} plus space as the subspace
\begin{align}
M_{\ell+1/2}^+(4):= \{f = \sum_{\substack{n\ge 0}}c_f(n)q^n\in M_{\ell+1/2}(4)~|~ c_f(n) = 0 \text{ if } (-1)^\ell n \equiv 2, 3 \pmod{4}\},
\end{align}
and put $S_{\ell+1/2}^+(4):=M_{\ell+1/2}^+(4)\cap S_{\ell+1/2}(4)$. Let $D$ be a fundamental discriminant (i.e. $D=1$ or is the discriminant of a quadratic field) such that $(-1)^{\ell}D>0$. Following Kohnen \cite[p.~251]{Kohenhalfintegralweight1980}, for $f(z)=\sum_{n\geq0}c_{f}(n)q^n\in M_{\ell+1/2}^{+}(4)$, we define its $D$-th Shimura lift as
\begin{align} 
    \mathcal{S}_{D}\left(\sum_{n\ge 0}c_f(n)q^n\right) := \frac{c_f(0)}{2}L_D(1-k)+\sum_{n\ge 1}\left(\sum_{d|n}\left(\frac{D}{d}\right)d^{k-1}c_f\left(|D|\frac{n^2}{d^2}\right)\right)q^n,\label{eq:defofshimuralift}
\end{align}
where $\legendre{D}{\cdot}$ is the Kronecker symbol.
It is known that  $\mathcal{S}_D$ maps $M_{\ell+1/2}^+(4)$ to $M_{2\ell}(1)$ and $S_{\ell+1/2}^+(4)$ to $S_{2\ell}(1)$, and commutes with the action of Hecke operators; see Kohnen \cite[Theorem 1]{Kohenhalfintegralweight1980} and Shimura \cite{Shimura1976}. 

Now we recall the Selberg identity on the Shimura lift. Let $\theta(z)=\sum_{n\in\mathbb{Z}}q^{n^2}\in M_{1/2}(4)$ be the Jacobi theta function. Selberg observed that for a normalized Hecke eigenform $f(z)\in M_{\ell}(1)$ with $a_f(1)=1$, the first Shimura lift provides the identity
\begin{align}
    \mathcal{S}_1(f(4z)\theta(z))=f(z)^2\in M_{2\ell}(1). \label{eq:selberg}
\end{align}

For a fundamental discriminant $D$ with $(-1)^{k} D>0$ with $k\ge4$ an integer, if one defines  
\begin{align}\mathscr{F}_D(z)&:=\Tr^D_1(G_{k,D}(z)^2)\in M_{2k}(1)\\
\mathscr{G}_{D}(z)&:=\frac32\left(1-\legendre{D}{2}2^{-k}\right)^{-1}\pr^+\Tr_4^{4D}(G_{k,4D}(4z)\theta(|D|z))\in M_{k+1/2}^{+}(4),\end{align} 
then Kohnen-Zagier  \cite[Proposition 3]{Kohnen-Zagier1981} proved the following generalization of \eqref{eq:selberg}:
\begin{align}\mathcal{S}_D(\mathscr{F}_D(z))=\mathscr{G}_D(z).\label{eq:liftkohnen-zagier}\end{align}
We must make several definitions for the above to make sense. The Eisenstein series $G_{k,D}$ and $G_{k,4D}$ are given by \cite[p.~185]{Kohnen-Zagier1981}
\begin{align}
    G_{k,D}(z):=&\frac{L_{D}(1-k)}{2}+\sum_{n=1}^{\infty}\left(\sum_{d\mid n}\legendre{D}{d}d^{k-1}\right)q^n\in M_{k}\left(|D|,\legendre{D}{\cdot}\right),\label{eq:G_k,D}\\
    G_{k,4D}(z):=&G_{k,D}(4z)-2^{-k}\legendre{D}{2}G_{k,D}(2z) \in M_{k}\left(4|D|,\legendre{D}{\cdot}\right), \label{eq:G_k,4D}
\end{align}
where $L_D(s)=\sum_{n\ge1}\legendre{D}{n}n^{-s}$. The operator $\pr^{+}$ is the projection from $M_{\ell+1/2}(4)$ to $M_{\ell+1/2}^{+}(4)$ given by \cite[p.~195]{Kohnen-Zagier1981}
\begin{align}
    (\pr^{+}g)(z)=\frac{1-(-1)^\ell i}{6}(\Tr^{16}_4Vg)(z)+\frac{1}{3}g(z),\label{eq:defofprojection}
\end{align} 
where $V(g)(z)=g(z+\frac{1}{4})=g(z)|_{k+1/2}\begin{bsmallmatrix}
    4 & 1\\ 0& 4
\end{bsmallmatrix}$, using the notation of \eqref{eq:tracemap} and \eqref{eq:slashoperator}. 
Additionally, for $N\mid M$, $\Tr_N^M$  is the trace map 
\begin{align} \label{eq:tracemap}
    \Tr_N^M: M_{m}(M)\rightarrow M_m(N),\quad g\mapsto \sum_{\gamma\in\Gamma_0(M)\backslash\Gamma_0(N)}g|_m\gamma,
\end{align}
where for any real number $m$ and $\gamma=\begin{bsmallmatrix}
    a&b\\c&d
\end{bsmallmatrix}\in \Gl_2^+(\BR)$ we define the slash operator \cite[Theorem 7.1]{Cohen'smodularformC_k}
\begin{align} \label{eq:slashoperator}
\left(g|_m\gamma\right)(z)=\det(\gamma)^{m/2} (cz+d)^{-m} g\left(\frac{az+b}{cz+d}\right).
\end{align}

On the other hand, the Selberg identity \eqref{eq:selberg} for the first Shimura lift has been generalized to the setting of Rankin-Cohen brackets. Let us first introduce the definition of Rankin-Cohen brackets for  modular forms.
\begin{definition}\label{def:rankincohen}
Let $f(z)\in M_{a}(\Gamma)$ and $g(z)\in M_b(\Gamma)$ be modular forms for some congruence subgroup $\Gamma$ of weights $a$ and $b$, respectively. For an nonnegative integer $e$, we define the $e$-th Rankin-Cohen bracket as 
\begin{align}
    [f(z),g(z)]_e := \sum_{r=0}^e (-1)^r\binom{e+a-1}{e-r}\binom{e+b-1}{r}f(z)^{(r)}g(z)^{(e-r)},\label{eq:defofrankin-cohenbracket}
\end{align}
where $f(z)^{(r)}$ is the $r$-th normalized derivative $f(z)^{(r)}:=\frac{1}{(2\pi i)^r}\frac{d^r f(z)}{dz^r}$ of $f$. Here $a,b$ can be in $\frac{1}{2}\mathbb{Z}$ and the binomial coefficients are defined through gamma functions. Moreover, $[f,g]_e\in M_{a+b+2e}(\Gamma)$ and $[f,g]_e\in S_{a+b+2e}(\Gamma)$ for $e>1$; see \cite[Theorem 7.1]{Cohen'smodularformC_k}. We remark that the Rankin-Cohen bracket defined in Zagier \cite[(73)]{Zagier1976} is related to \eqref{eq:defofrankin-cohenbracket} through $F_{e}^{(a,b)}(f(z),g(z))= (-2\pi i)^e e![f(z),g(z)]_e$; see  \cite[(1.1)]{Xue2203}.
\end{definition}

Choie-Kohnen-Zhang \cite{choie2024rankincohen} and Xue \cite{Xue-Selbergidentity} independently showed that if $k \ge 4$ is an even integer, $f(z)\in M_{k}(1)$ is a normalized Hecke eigenform, and $e$ is a nonnegative integer, then
\begin{equation}
    \mathcal{S}_1([f(4z), \theta(z)]_e)=\frac{\binom{k+e-1}{e}}{\binom{k+2e-1}{2e}}[f(z),f(z)]_{2e}.\label{eq:levelonerankincohenlift}
\end{equation}
\noindent Note that \eqref{eq:levelonerankincohenlift} was also proved in \cite[Proposition B1]{Popa2011} when $f$ is an Eisenstein series.  Let $k\geq4$ and $e>0$ be integers with $\ell=k+2e$ and let $D$ be a fundamental discriminant such that $(-1)^{\ell}D>0$. 
We introduce functions
\begin{align} 
    \mathcal{F}_{D,k,e}(z):=&\Tr^D_1([G_{k,D}(z),G_{k,D}(z)]_{2e})\in S_{2\ell}(1),\label{eq:F_D,k,e}\\
     \mathcal{G}_{D, k,e}(z):=&\frac32\left(1-\legendre{D}{2}2^{-k}\right)^{-1}\pr^+\Tr^{4D}_4\left[G_{k,4D}(z),\theta(|D|z)\right]_e \in S^+_{\ell+1/2}(4).\label{eq:G_D,k,e}
\end{align}
Note that both $\mathcal{F}_{D,k,e}(z)$ are $\mathcal{G}_{D, k,e}(z)$ are cuspforms, since $e>0$.
Now, we state our first main result, which can be viewed as a combination of \eqref{eq:liftkohnen-zagier} and \eqref{eq:levelonerankincohenlift}.
\begin{theorem}\label{thm:lift}
Let $D$ be an odd fundamental discriminant such that $(-1)^\ell D>0$ and let $k\ge 4$ and $e>0$ be integers such that $k+2e=\ell$. Then we have the identity
\begin{align}
\mathcal{S}_D\left(\mathcal{G}_{D,k,e}\right)=|D|^e\frac{\binom{k+e-1}{e}}{\binom{k+2e-1}{2e}}\mathcal{F}_{D,k,e}.\label{eq:liftofG}
\end{align}
\end{theorem}
We have required that $e>0$ because the case $e=0$ is exactly \eqref{eq:liftkohnen-zagier}. Our next main result concerns the  nonvanishing of twisted central values of $L$-functions associated to Hecke eigenforms. Before stating the precise result, let us first introduce some notation.

\begin{definition} \label{def:subspace}
Let $D$ be a fundamental discriminant such that $(-1)^{\ell}D>0$.
\begin{enumerate}
\item Let $S_{2\ell}^{0,D}(1)$ denote the subspace of $S_{2\ell}(1)$ generated by normalized Hecke eigenforms $f$ with nonzero central twisted $L$-values $L(f,D,\ell)$, where $L(f,D,s)=\sum_{n\ge1}\legendre{D}{n}a_f(n) n^{-s}$ is the $L$-function of $f$ twisted by $\legendre{D}{\cdot}$. We write $S^{-,D}_{2\ell}(1)$ for the orthogonal complement of $S_{2\ell}^{0,D}(1)$, which is spanned by Hecke eigenforms with vanishing central twisted $L$-values.
\item Let $S_{\ell+1/2}^{0,D}(4)$ be the subspace of $S^+_{\ell+1/2}(4)$ generated by Hecke eigenforms $g=\sum_{n\geq1}c_g(n)q^n$ with $c_g(|D|)\ne0$. We write $S^{-,D}_{\ell+1/2}(4)$ for the orthogonal complement of $S_{\ell+1/2}^{0,D}(4)$, which is spanned by Hecke eigenforms $g=\sum_{n\geq1}c_g(n)q^n$ with $c_g(|D|)=0$.
\end{enumerate}
\end{definition}

The twisted $L$-function $L(f,D,s)$, originally defined for $\re(s)\gg0$, can be analytically continued to the whole complex plane, and for a Hecke eigenform $f\in S_{2\ell}(1)$ satisfies \cite[Lemma 9.2]{Ono2004}:
\begin{align}
 \Lambda(f,D,s)=(-1)^{\ell} \legendre{D}{-1} \Lambda(f,D,2\ell-s),   
\end{align}
where $\Lambda(f,D,s)=(2\pi)^{-s}\Gamma(s) L(f,D,s)$ is the completed twisted $L$-function of $f$. Since $\legendre{D}{-1}$ is the sign of $D$, the assumption $(-1)^{\ell}D>0$ implies that the functional equation for $L(f,D,s)$ has a positive sign. Therefore, the subspace $S_{2\ell}^{0,D}(1)$ in Definition \ref{def:subspace} (1) is not trivially zero. It is speculated that the central $L$-value $L(f,D,\ell)$ is nonvanishing for every Hecke eigenform $f\in S_{2\ell}(1)$. Thus, it is believed that $S_{2\ell}(1)=S^{0,D}_{2\ell}(1)$ for every fundamental discriminant $D$. For further discussion, see Section \ref{sec:discussion}.

Our second main result gives an explicit construction of a set of generators for the subspaces $S_{2\ell}^{0,D}(1)$ and $S_{\ell+1/2}^{0,D}(4)$. We hope this result would help investigate the aforementioned speculation on the nonvanishing of twisted central $L$-values. Furthermore, we prove that the $D$-th Shimura lift $\mathcal{S}_D$ gives an isomorphism between $S_{\ell+1/2}^{0,D}(4)$ and $S_{2\ell}^{0,D}(1)$, which generalizes Kohnen's results  \cite[Theorem 2]{Kohenhalfintegralweight1980} and \cite[Proposition 3.3]{Xue-Selbergidentity}.
\begin{theorem}\label{thm:gspan}
Let $D$ be an odd fundamental discriminant with $(-1)^{\ell}D>0$. Then 
\begin{align}
    S_{\ell+1/2}^{0,D}(4)=\Span\{\mathcal{G}_{D,k,e}\}_{k+2e=\ell},\quad \text{and}\quad S_{2\ell}^{0, D}(1)=\Span\{\mathcal{F}_{D,k,e}\}_{2k+4e=2\ell},
\end{align}
where $k\ge4$ and $e>0$. Additionally, the restricted $D$-th Shimura lift
\begin{align}
    S_D: S_{\ell+1/2}^{0,D}(4) \rightarrow S_{2\ell}^{0, D}(1)
\end{align}
is an isomorphism.
\end{theorem}

We assume $D$ to be odd throughout the paper in order to avoid the technical complications caused by even $D$, although we believe our results continue to hold in this case.

This paper is organized as follows. 
Section \ref{sec:nonvanishing} discusses the main results of this paper.  The proof of Theorem \ref{thm:lift} is based on the same idea as the proof of \eqref{eq:levelonerankincohenlift} (see \cite{choie2024rankincohen} and \cite{Xue-Selbergidentity}), but requires explicit computations of the Fourier coefficients of both sides of \eqref{eq:liftofG}. Most of the technical details required for the proof of Theorem \ref{thm:lift} are presented in Section \ref{sec:fourier}. Based on the Petersson inner product formalae for $\mathcal{F}_{D,k,e}$ and $\mathcal{G}_{D,k,e}$ derived in Section \ref{sec:rankin}, we explicitly construct a spanning set for $S_{2\ell}^{0,D}(1)$ (Proposition \ref{prop:fspan}). We then show that the $D$-th Shimura lift is an isomorphism from $S^{0,D}_{\ell+1/2}(4)$ to $S_{2\ell}^{0,D}(1)$ (Proposition \ref{prop:isomorphism}). Finally, using these results 
we prove Proposition \ref{prop:GDkespan}, explicitly constructing a spanning set for $S_{\ell+1/2}^{0,D}(4)$ and finishing the proof of Theorem \ref{thm:gspan}.

The remaining sections are dedicated to proofs of the results needed in Section \ref{sec:nonvanishing}. Section \ref{sec:projection} proves an alternate formula for $\mathcal{G}_{D,k,e}$ which we use to compute its Fourier coefficients in Section \ref{sec:fourier}. Section \ref{sec:Eisenstein} recalls the theory of Eisenstein series, which will be useful to the Fourier development of $\mathcal{F}_{D,k,e}$ and $\mathcal{G}_{D,k,e}$ in Section \ref{sec:fourier}.  Assuming those two sections, Section \ref{sec:rankin} derives Petersson inner product formulae for $\mathcal{F}_{D,k,e}$ and $\mathcal{G}_{D,k,e}$ via the Rankin-Selberg convolution. In Section \ref{sec:fourier}, we carry out the computations of Fourier coefficients for Theorem \ref{thm:lift}. Section \ref{sec:discussion} discusses the relationship between these results and their potential applications to the nonvanishing of twisted central $L$-values of Hecke eigenforms in $S_{2\ell}(1)$.

\section{Selberg identity and spanning sets of subspaces}\label{sec:nonvanishing}

This section proves our main results, assuming the necessary results to be proved later. We begin by proving Theorem \ref{thm:lift}, a generalization of the Selberg identity.

\begin{proof}[Proof of Theorem \ref{thm:lift}]

Recall that $\mathcal{G}_{D,k,e}$ \eqref{eq:G_D,k,e} and $\mathcal{F}_{D,k,e}$ \eqref{eq:F_D,k,e} are cuspforms.  Write 
$$\mathcal{S}_D(\mathcal{G}_{D,k,e}(z))=\sum_{n\geq1}g_{D,k,e}(n)q^n\quad{\rm and}\quad\mathcal{F}_{D,k,e}(z)=\sum_{n\geq1}f_{D,k,e}(n)q^n.$$ Comparing the Fourier coefficients $f_{D,k,e}(n)$ and $g_{D,k,e}(n)$ that are respectively given by Lemma \ref{lem:ffouriercoefficient} and Lemma \ref{lem:gfouriercoefficients}, it suffices to show for each nonnegative integer pair $(a_1, a_2)$ with $a_1+a_2=n|D_1|$ that
    \begin{align}
        &\binom{k+e-1}{e}\sum_{r=0}^{2e}(-1)^ra_1^ra_2^{2e-r}\binom{2e+k-1}{2e-r}\binom{2e+k-1}{r}\\
        &=\binom{k+2e-1}{2e}\sum_{r+s=e}(-1)^r\binom{k+e-1}{s}\binom{e-1/2}{r}4^r(a_2-a_1)^{2s}(a_1a_2)^r.\label{eq:identitycom}
    \end{align}

Without loss of generality we may assume that $R\leq S$ and compare the coefficients of the monomial $a_1^Ra_2^S$ of the two sides of \eqref{eq:identitycom}. The $a_1^Ra_2^S$-coefficient on the left hand side of \eqref{eq:identitycom} is 
  \begin{align}
      (-1)^R\binom{k+e-1}{e}\binom{2e+k-1}{2e-R}\binom{2e+k-1}{R},
  \end{align}
  and the right hand side of \eqref{eq:identitycom} has $a_1^Ra_2^S$-coefficient 
  \begin{align}
      &\binom{k+2e-1}{2e}\sum_{r=0}^R(-1)^r\binom{k+e-1}{e-r}\binom{e-1/2}{r}4^r\binom{2e-2r}{R-r}(-1)^{R-r}\\=&(-1)^R\binom{k+2e-1}{2e}\sum_{r=0}^R\binom{k+e-1}{e-r}\binom{e-1/2}{r}4^r\binom{2e-2r}{R-r}.
  \end{align}
Using Lemma \ref{lem:combinatorialidentity}, we finish the proof of Theorem \ref{thm:lift}.    
\end{proof}
\begin{lemma}\label{lem:combinatorialidentity}
  Let $R\leq e$ be nonnegative and $k\geq4$. Then we have the following identity
    \begin{equation*}
        \binom{k+e-1}{e}\binom{k+2e-1}{2e-R}\binom{k+2e-1}{R}=\binom{k+2e-1}{2e}\sum_{r=0}^R 4^r\binom{k+e-1}{e-r}\binom{e-1/2}{r}\binom{2e-2r}{R-r},
    \end{equation*}
    where fractional binomial coefficients are defined by the $\Gamma$ function.
\end{lemma}

\begin{proof}
    We reproduce the proof of \cite[Proposition 2.1]{Xue-Selbergidentity}. By definition we have
    \begin{equation*}
        \binom{e-1/2}{r}=\frac{\Gamma(e+1/2)}{\Gamma(r+1)\Gamma(e+1/2-r)}.
    \end{equation*}
    By Legendre's duplication formulas we have
    \begin{equation*}
        \Gamma(e+1/2)=\frac{(2e)!}{4^ee!}\sqrt{\pi}, \qquad \Gamma(e-r+1/2)=\frac{(2(e-r))!}{4^{e-r}(e-r)!}\sqrt{\pi}.
    \end{equation*}
    These together yield
    \begin{equation*}
        \binom{e-1/2}{r}=\frac{(2e)!4^{e-r}(e-r)!\sqrt{\pi}}{r!4^ee!(2(e-r))!\sqrt{\pi}}=\frac{(2e)!4^{r}(e-r)!}{r!e!(2(e-r))!},
    \end{equation*}
    which yields the following formula for each term on the left hand side
    \begin{align*}
        4^r\hspace{-1pt}\binom{k+2e-1}{2e}\hspace{-2pt}\binom{k+e-1}{e-r}\hspace{-2pt}\binom{e-1/2}{r}\hspace{-2pt}\binom{2e-2r}{R-r}\hspace{-1pt}=\hspace{-1pt}\frac{(k+2e-1)!(k+e-1)!}{(k-1)!\hspace{-1pt}(k+r-1)!\hspace{-1pt}(R-r)!\hspace{-1pt}(2e-R-r)!e!r!}.
    \end{align*}
    The right hand side expands in to
    \begin{equation*}
        \binom{k+e-1}{e}\binom{k+2e-1}{2e-R}\binom{k+2e-1}{R}=\frac{(k+e-1)!(k+2e-1)!(k+2e-1)!}{e!(k-1)!(2e-R)!(k+R-1)!R!(k+2e-1-R)!}.
    \end{equation*}
    If we cancel $(k+e-1)!(k+2e-1)!$ from both sides, and multiply by $R!(k+2e-R-1)$, we see that it suffices to show
    \begin{equation*}
        \sum_{r=0}^R \binom{R}{R-r}\binom{k+2e-R-1}{k+r-1}=\binom{k+2e-1}{k+R-1}.
    \end{equation*}
    After applying the involution $r\mapsto R-r$, this is then the Vandermonde's identity \cite[p.11]{Riordancombinatoriaidentities}
    \begin{equation*}
        \sum_{j=0}^t \binom{n}{j}\binom{m}{t-j}=\binom{n+m}{t}
    \end{equation*}
    for the case of $n=R$, $m=k+2e-R-1$, and $t=k+R-1$.
\end{proof}
 We now build toward the proof of Theorem \ref{thm:gspan}. We begin by showing that the $D$-th Shimura lift gives rise to an isomorphism between $S_{\ell+1/2}^{0,D}(4)$ and $S_{2\ell}^{0,D}(1)$, which is a generalization of \cite[Theorem 2]{Kohenhalfintegralweight1980} for $D=1$.
\begin{proposition}\label{prop:isomorphism}
    Let $D$ be an odd fundamental discriminant with $(-1)^{\ell}D>0$. Then the $D$-th Shimura lift $\mathcal{S}_D$ restricts to an isomorphism $S_{\ell+1/2}^{0,D}(4) \rightarrow S_{2\ell}^{0,D}(1)$ for all $\ell \geq 6$. 
\end{proposition}
\begin{proof}

Recall that by \cite[Theorem 1]{Kohenhalfintegralweight1980} or \cite[p.~182]{Kohnen-Zagier1981}, if  $g=\sum_{n\ge1} c_g(n)q^n\in S^+_{\ell+1/2}(4)$ is a Hecke eigenform and $f\in S_{2\ell}(1)$ is the normalized Hecke eigenform corresponding to $g$, then $\mathcal{S}_D(g)=c_g(|D|) f$. This means that $\mathcal{S}_D$ is a monomorphism when restricted to $S^{0,D}_{\ell+1/2}(4)$.
Thus, in order to show $\mathcal{S}_D$ restricts to an isomorphism from $S_{\ell+1/2}^{0,D}(4)$ to $ S_{2\ell}^{0,D}(1)$  it suffices to show that $\dim S_{\ell+1/2}^{0,D}(4)=\dim S_{2\ell}^{0,D}(1)$. 

Note that  $\dim S_{2\ell}^{0,D}(1)$ is the number of Hecke eigenforms in $S_{2\ell}(1)$ with nonzero central twisted $L$-value, and $\dim S_{\ell+1/2}^{0,D}(4)$ is the number of Hecke eigenforms in $S^+_{\ell+1/2}(4)$ with nonzero $|D|$-th Fourier coefficient. According to \cite[Theorem 1]{Kohnen-Zagier1981}, these two nonvanishing conditions are the same under the Shimura correspondence, thus we conclude that $\dim S_{\ell+1/2}^{0,D}(4)=\dim S_{2\ell}^{0,D}(1)$. 
\end{proof}
\begin{remark}
In the $\ell = 5,7$ case, the space of cuspforms $S_{2\ell}(1)$ is empty, and so is the space $S^+_{\ell+1/2}(4)$. So this proposition is trivially true. 
\end{remark}
We now construct an explicit spanning set for $S_{2\ell}^{0,D}(1)$. Before doing so, we need to introduce the Period of a modular form. For $f\in S_{2\ell}(1)$ and $0\leq t\leq 2\ell-2$, the $t$-th Period of $f$ is given by 
\begin{align}r_t(f):=\frac{t!}{(-2\pi i)^{t+1}}L(f,t+1).\label{eq:defofperiods}\end{align}
Here the $L$-series of $f(z)=\sum_{n\geq1}a_nq^n$ is $L(f,s)=\sum_{n\geq1}a_nn^{-s}$, which converges for $\re(s)\gg0$ and can be extended analytically to the whole complex plane; for details, see \cite{Manin1973}.  
\begin{proposition}\label{prop:fspan}
    The set $\{\mathcal{F}_{D, k, e}\}_{2k+4e=2\ell}$ for $1\leq e \leq \lfloor \frac{\ell-4}{2}\rfloor$ spans $S_{2\ell}^{0,D}(1)$, for all $\ell \geq 6$. %
\end{proposition}

\begin{proof}

By Proposition \ref{prop:FdkeInnerProduct}, we know that if $g\in S_{2\ell}^{-, D}(1)$ then $g$ is orthogonal to the subspace of $S_{2\ell}^D(1)$ spanned by $\{\mathcal{F}_{D, k, e}\}_{2k+4e=2\ell}$. So it suffices to show that the orthogonal complement of the span of $\{\mathcal{F}_{D, k, e}\}_{2k+4e=2\ell}$ is contained in $S_{2\ell}^{-, D}(1)$.

We will show that any modular form $G=\sum_j c_jg_j$ which is a linear combination of normalized Hecke eigenforms in $g_j \in S_{2\ell}^{0, D}(1)$ such that $\langle G, \mathcal{F}_{D,k,e} \rangle =0$ for all  $\mathcal{F}_{D, k, e}$ must be zero.

Note that Proposition \ref{prop:FdkeInnerProduct} and \eqref{eq:defofperiods} imply that 
\begin{align}
     \langle \mathcal{F}_{D,k,e},g_j\rangle
    =\frac{1}{2}\frac{\Gamma(2k+4e-1)\Gamma(k+2e)}{(2e)!(4\pi)^{2k+4e-1}\Gamma(k)}\frac{L_{D}(1-k)}{L_D(k)}\frac{(-2\pi i)^{2k+2e-1}}{(2k+2e-2)!}L(g_j,D,k+2e)r_{2k+2e-2}(g_j).
\end{align}
Thus, the orthogonality condition $ \langle G,\mathcal{F}_{D,k,e}\rangle=0$ is equivalent to
\begin{align}
    \sum_{j}c_jL(g_j,D,k+2e)r_{2k+2e-2}(g_j)=0.\label{eq:orthogonality}
\end{align}
Following an idea from the proof of \cite[Theorem 1]{Kohnen2005}, we define another form in $S_{2\ell}(1)$ by
$$F=\sum_j c_j L(g_j,D,k+2e)g_j.$$ 
Hence \eqref{eq:orthogonality} implies that
\begin{align}
    r_{2k+2e-2}(F)=\sum_jc_jL(g_j,D,k+2e)r_{2k+2e-2}(g_j)=0.
\end{align}
As $1\leq e \leq \lfloor \frac{\ell-4}{2}\rfloor$ and $k+2e = \ell$, we have $\ell-2 \geq k \geq 4$. Then $t=2k+2e-2$ ranges through all even values $\ell+2\leq t\leq 2\ell-4$, so $r_{t}(F)=0$ for all even $\ell+2\leq t\leq 2\ell-4$.
As a result of the following lemma, we have $F=0$. Since $L(g_j,D,k+2e)\neq0$ as $g_j\in S_{2\ell}^{0,D}(1)$, we must have $c_j=0$ for all $j$, and thus $G=0$.

\end{proof}

\begin{lemma}\label{lem:periodvanishing}
    Let $F\in S_{2\ell}(1)$ and $\ell\geq 6$, and let $r_t(F)$ be the $t$-th Period of $F$. If  $r_t(F)=0$ for all even $t$ such that $\ell+2\leq t \leq 2\ell-4$, then $F=0$.
\end{lemma}
\begin{proof}We follow the idea of \cite{XueRankin-Cohen}.
    By the Eichler-Shimura theory \cite[Proposition 2.3 (b)]{Manin1973} and \cite[Remark 2.4]{XueRankin-Cohen}, we know that $F=0$ if and only if $r_t(F)=0$ for all even $2\leq t\leq 2\ell-4$. 
    By the Eichler-Shimura relation \begin{align}\label{eq:eichlershimurafunctionaleq}
        r_t(F)+(-1)^{t}r_{2\ell-2-t}(F)=0,
    \end{align}
    and the assumption that $r_t(F)=0$ for all even $\ell+2\leq t \leq 2\ell-4$, we know that $r_t(F)=0$ also for all even $2\leq t\leq \ell-4$. To show that the Periods $\ell-4 < t < \ell+2$ are zero, we split into cases based on the parity of $\ell$.
    
    \begin{enumerate}
        \item If $\ell$ is even, it suffices to show that $r_\ell(F)=r_{\ell-2}(F)=0$.
    Since $\ell$ is even, by \eqref{eq:eichlershimurafunctionaleq} 
    \begin{align}\label{eq:eichellminustwo}
        r_\ell(F)+r_{\ell-2}(F)=0.
    \end{align}
    Substituting $t=\ell-2$ into the Eichler-Shimura relation 
    \begin{align}\label{eq:eichlershimura2t}
        (-1)^tr_t(F)+\sum_{\substack{0\leq m\leq t\\ m\equiv 0\pmod2}}\binom{t}{m}r_{2\ell-2-t+m}(F)+\sum_{\substack{0\leq m\leq 2\ell-2-t\\ m\equiv t\pmod2}}\binom{2\ell-2-t}{m}r_{m}(F)=0
    \end{align}
    and noting that $r_0(F)+r_{2\ell-2}(F)=0$,  we obtain 
    \begin{align}
    \left(\binom{\ell}{2}+1\right)r_{\ell-2}(F)+2r_{\ell}(F)&=0.
    \end{align}
    This equation, along with \eqref{eq:eichellminustwo} implies that $r_\ell(F)=r_{\ell-2}(F)=0$ for $\ell\ge6$.
    \item If $\ell$ is odd, it suffices to show that $r_{\ell-3}(F)=r_{\ell-1}(F)=r_{\ell+1}(F)=0$. Substituting $t = \ell -1$ into \eqref{eq:eichlershimura2t}, we get
    \[3r_{\ell-1}(F) + \binom{\ell-1}{2} r_{\ell+1}(F) + \binom{\ell-1}{\ell-3} r_{\ell-3}(F) = 0.\]
    Since $\binom{\ell-1}{2} = \binom{\ell-1}{\ell-3}$, and we know by \eqref{eq:eichlershimurafunctionaleq} that $ r_{\ell-3}(F)+r_{\ell+1}(F)=0$, we conclude that $r_{\ell-1}(F) = 0$. 
    Substituting $t=\ell+1$ into \eqref{eq:eichlershimura2t} yields
    \begin{align}
        2r_{\ell-3}(F)+\binom{\ell+1}{2}r_{\ell-1}(F)+\left(1+\binom{\ell+1}{4}\right)r_{\ell+1}(F)=0,
    \end{align}
    and noting that $r_{\ell-1}(F)=0$, we conclude by \eqref{eq:eichlershimurafunctionaleq} that
    \begin{align}
        r_{\ell-3}(F)=r_{\ell+1}(F)=0.
    \end{align}
    \end{enumerate}
    This finishes the proof. 
\end{proof}

Finally, we construct a spanning set for $S_{\ell+1/2}^{0,D}(4)$ and finish the proof of Theorem \ref{thm:gspan}. 
\begin{proposition}\label{prop:GDkespan}
    The set $\{\mathcal{G}_{D,k,e}\}_{k+2e=\ell}$ for $1\leq e \leq \lfloor \frac{\ell-4}{2}\rfloor$ spans the subspace $S_{\ell+1/2}^{0,D}(4)$, for all $\ell \geq 6$.  
\end{proposition}
\begin{proof}
For a Hecke eigenform $g \in S_{\ell+1/2}^{-, D}(4)$, we have $\langle g,\mathcal{G}_{D,k,e}\rangle=0$  by Proposition \ref{prop:RankinSelberghalfintegral}. So  $g$ is orthogonal to $\Span\{\mathcal{G}_{D,k,e}\}_{k+2e=\ell}$ and thus
  \begin{align}\Span\{\mathcal{G}_{D,k,e}\}_{k+2e=\ell}\subseteq S_{\ell+1/2}^{0,D}(4).\label{eq:condition1}\end{align}
  Note that Theorem \ref{thm:lift} implies that 
  \begin{align}
      \dim \Span\{\mathcal{G}_{D,k,e}\}_{k+2e=\ell}\geq \dim\Span\{\mathcal{F}_{D,k,e}\}_{k+2e=\ell}.\label{eq:condition2}
  \end{align}
By Propositions  \ref{prop:isomorphism} and \ref{prop:fspan}, we have 
  \begin{align}
      \quad \Span\{\mathcal{F}_{D,k,e}\}_{k+2e=\ell}=S_{2\ell}^{0,D}(1) \quad{\rm and}\dim S_{2\ell}^{0,D}(1)=\dim S_{\ell+1/2}^{0,D}(4).\label{eq:condition3}
  \end{align}
    Now \eqref{eq:condition1}, \eqref{eq:condition2} and \eqref{eq:condition3} together imply that $\dim \Span\{\mathcal{G}_{D,k,e}\}_{k+2e=\ell} \geq  \dim S_{\ell+1/2}^{0,D}(4)$. So we conclude that $\Span\{\mathcal{G}_{D,k,e}\}_{k+2e=\ell} = S_{\ell+1/2}^{0,D}(4)$.
\end{proof}

Combining Propositions \ref{prop:isomorphism}, \ref{prop:fspan} and \ref{prop:GDkespan}, we complete the proof of Theorem \ref{thm:gspan}.

\section{Projection}\label{sec:projection}
In this section we prove an alternate formula for $\mathcal{G}_{D,k,e}$ \eqref{eq:G_D,k,e}:
\begin{equation} 
\mathcal{G}_{D, k, e}(z)=\Tr^{4D}_4 [G_{k,D}(4z), \theta(|D|z)]_e.
\end{equation}
A similar formula is implicit in equations (6) and (7) in \cite{Kohnen-Zagier1981}.  
This formula allows us to compute the Fourier coefficients (Proposition \ref{prop:FourierexofGDe}).

We need to introduce some notation and facts needed for the proof of Lemma \ref{lem:vswap}. Let
\begin{align}
    \mathbb{P}^1(\mathbb{Z}/N\mathbb{Z})=\{(a:b):a,b\in\mathbb{Z}/N\mathbb{Z},\,\gcd(a,b,N)=1\}/\sim
\end{align}
be the projective line over $\mathbb{Z}/N\mathbb{Z}$, where $(a:b)\sim(a':b')$ if there exists $u\in(\mathbb{Z}/N\mathbb{Z})^{\ast}$ such that $a=ua', b=ub'$. It is known that there is a bijection between $\Gamma_0(N)\backslash\Sl_2(\mathbb{Z})$ and $\mathbb{P}^1(\mathbb{Z}/N\mathbb{Z})$, which sends a coset representative $\begin{bsmallmatrix}
    a &b\\c&d
\end{bsmallmatrix}$ to the class $(c:d)$ in $\mathbb{P}^1(\ZZ/N\mathbb{Z})$, see \cite[Proposition 3.10]{Steinbook}. For future reference, we prove a result on coset representatives of certain quotients of congruence subgroups.

\begin{lemma}\label{lem:cosetreps}
    Let $N\in \mathbb{N}$ and let $S\in \mathbb{N}$ be squarefree with $(N,S)=1$. Then  \begin{equation*}
      \left\{  \begin{bmatrix}
            1 & 0 \\ NS_1 & 1
        \end{bmatrix}\begin{bmatrix}
            1 & \mu \\ 0 & 1
        \end{bmatrix}\quad :\quad S_1\mid S,\quad \mu~\text{mod}~S_2\right\}
    \end{equation*}
    (meaning $\mu$ ranges over all values mod $S_2$) is a set of coset representatives for $\Gamma_0(NS)\backslash \Gamma_0(N)$.
\end{lemma}
\begin{proof}
    We use the description of the cosets $\Gamma_0(NS)\backslash \Gamma_0(N)$ given in (\cite{GZ-86}, p. 276), which says that a coset $\Gamma_0(NS)\begin{bsmallmatrix}
        a & b \\ c & d
    \end{bsmallmatrix}$ is determined by the value of $S_1=(c,S)$ and the value of $c^*d$ mod $S/S_1$. Given this, it suffices to show that
    \begin{equation*}
        \begin{bmatrix}
            a & b \\ c & d
        \end{bmatrix}\coloneqq\begin{bmatrix}
            1 & \mu \\ NS_1 & \mu N S_1+1
        \end{bmatrix}=\begin{bmatrix}
            1 & 0 \\ NS_1 & 1
        \end{bmatrix}\begin{bmatrix}
            1 & \mu \\ 0 & 1
        \end{bmatrix}
    \end{equation*}
    ranges over all possible values of $\text{gcd}(c,S)$ and all possible values of $c^*d\mod S_2$, where $c^*$ is an integer representative of $c^{-1}~\text{mod}~S_2$. It clearly ranges over all values of $\text{gcd}(c,S)$, since $S_1$ is an arbitrary divisor of $S$ and $(N,S)=1$. Moreover, the values of $c^*d\equiv (NS_1)^*(\mu N S_1)+1\equiv \mu+1~\text{mod}~S_2$ range over all values mod $S_2$ because $\mu$ is arbitrary.
\end{proof}
\begin{lemma}\label{lem:vswap}
   Let $\ell\ge1$ be an integer and $D$ be odd. We have $V\Tr^{4D}_4g=\Tr^{16D}_{16}Vg$ for all $g\in M_{\ell+1/2}(4|D|)$.
\end{lemma}
\begin{proof}
    We first remark that by direct calculation, $Vg\in M_{\ell + \frac{1}{2}}(16|D|)$, so $\Tr^{16D}_{16}Vg$ is well-defined. Note that applying the fixed set of cosets for $\Gamma_0(4D)\backslash \Gamma_0(4)$ and $\Gamma_0(16D)\backslash \Gamma_0(16)$ given by Lemma \ref{lem:cosetreps} to $N=4,16$ and $S=|D|$, we have the following explicit formulas (see \eqref{eq:slashoperator} for the definition of slash operators) 
    \begin{align}
        V\Tr^{4D}_4g(z)&=\sum_{D_1D_2=D}\sum_{\mu~\text{mod}~|D_2|} g(z)|_\ell \gamma_{D_1,\mu}\begin{bsmallmatrix}
            1 & \frac{1}{4} \\ 0 & 1
        \end{bsmallmatrix},
        \\
        \Tr^{16D}_{16}Vg(z)&=\sum_{D_1D_2=D}\sum_{\mu~\text{mod}~|D_2|} g(z)|_\ell \begin{bsmallmatrix}
            1 & \frac{1}{4} \\ 0 & 1
        \end{bsmallmatrix}\gamma'_{D_1,\mu},
    \end{align}
    where
    \begin{equation*}
       \gamma_{D_1,\mu}=\begin{bmatrix}
            1 & 0 \\ 4|D_1| & 1
        \end{bmatrix}\begin{bmatrix}
            1 & \mu \\ 0 & 1
        \end{bmatrix}\quad{\rm and}\quad\gamma'_{D_1,\mu}=\begin{bmatrix}
            1 & 0 \\ 16|D_1| & 1
        \end{bmatrix}\begin{bmatrix}
            1 & \mu \\ 0 & 1
        \end{bmatrix}.
    \end{equation*}
    And the outer sums are over all factorizations of $D$ into a product of fundamental discriminants $D_1, D_2$. Therefore, to prove the desired equality it suffices to show that the set of cosets
    \begin{equation*}
        \left\{\Gamma_0(4|D|)\begin{bmatrix}
            1 & 1/4 \\ 0 & 1
        \end{bmatrix}\gamma'_{D_1,\mu}\begin{bmatrix}
            1 &1/4\\0 &1 \end{bmatrix}^{-1}:\,D_1D_2=D,\quad\mu~\text{mod}~{|D_2|}\right\}
    \end{equation*}
    is a system of representatives of $\Gamma_0(4|D|)\backslash \Gamma_0(4)$.

    Computing this conjugation, we have
    \begin{equation*}
        \begin{bmatrix}
            1 & 1/4 \\ 0 & 1
        \end{bmatrix}\gamma'_{D_1,\mu}\begin{bmatrix}
            1 & -1/4 \\ 0 & 1
        \end{bmatrix}=\begin{bmatrix}
1+4 {|D_{1}|} & -{|D_1|}+\left(1+4 {|D_1|}\right) \mu  
\\
 16 {|D_1|} & -4 {|D_1|}+16 {|D_{1}|} \mu +1 
\end{bmatrix} 
    .\end{equation*}
Using the bijection between $\Gamma_0(4|D|)\backslash \Gamma_0(4)$ and $\{(c:d)\in\mathbb{P}^1(\ZZ/4|D|\ZZ):4\mid c\}$, and fixing $|D_1|$, it suffices to show
    \begin{align}
        &\{(16|D_1|\,:\, 16|D_1|\mu-4|D_1|+1)\in \BP^1(\ZZ/4D\ZZ):\mu\text{ mod}|D_2|\},
        \\&=\{(c:d)\in \BP^1(\ZZ/4|D|\ZZ)\,:\,(c,4D)=4|D_1|\}.
    \end{align}
    Noting that the second set has size $|D_2|$ and that there are $|D_2|$ choices for $\mu$ in the first set, it suffices to show that no two choices of $\mu$ yield the same element of $\BP^1(\ZZ/4D\ZZ)$. To that end, suppose
    \begin{equation*}
        \alpha(16|D_1|:16|D_1|\mu-4|D_1|+1)=(16|D_1|:16|D_1|\mu'-4|D_1|+1)
    \end{equation*}
    for some $\alpha\in(\mathbb{Z}/4|D|\ZZ)^{\ast}$, which gives a system of congruences:
    \begin{equation*}
        \begin{cases}
            \alpha16|D_1|\equiv 16|D_1| \mod 4|D|, \\
            \alpha(16|D_1|\mu-4|D_1|+1)\equiv 16|D_1|\mu'-4|D_1|+1 \mod 4|D|.
        \end{cases}
    \end{equation*}
    The first equation gives $(\alpha-1)(16|D_1|)\equiv 0 \mod 4|D|$, and so $\alpha\equiv 1 \mod |D_2|$. Using this fact and taking the second equation mod $|D_2|$, we then get
    \begin{equation*}
        16|D_1|\mu-4|D_1|+1\equiv 16|D_1|\mu'-4|D_1|+1 \mod D_2.
    \end{equation*}
    Since $16|D_1|$ is invertible mod $|D_2|$, we obtain $\mu\equiv \mu'\mod |D_2|$, and the result follows.
\end{proof}
\begin{definition}
For $m\in\mathbb{N}$ and $f(z)=\sum_{n\geq0}a_f(n)q^n\in S_{k}(N,\chi)$ we define $U_mf$ by
\begin{align}
    U_mf(z)=\frac{1}{m}\sum_{v~{\rm mod}~m}f\left(\frac{z+v}{m}\right)=\sum_{n\geq0}a_f(mn)q^n.\label{eq:defofUmap}
\end{align}
\end{definition}
\noindent Equivalently, we may write \eqref{eq:slashoperator}
\begin{align}U_mf(z)=m^{k/2-1}\sum_{v\text{ mod }m}f(z)\bigg|_k\begin{bmatrix}
    1 & v\\ 0& m
\end{bmatrix}.\label{eq:defofUslash}\end{align} 
We need the following two simple observations.
\begin{lemma}\label{lem:U2ofGkd}
  Let $U_2$ be the operator defined in \eqref{eq:defofUmap}. Then 
    \begin{align}
        U_2G_{k,D}(z)=\left(1+2^{k-1}\legendre{D}{2}\right)G_{k,D}(z)-2^{k-1}\legendre{D}{2}G_{k,D}(2z).\label{eq:U2formular}
    \end{align}
\end{lemma}
\begin{proof}
Recall that $G_{k,D}$ is defined by \eqref{eq:G_k,D}:
 \begin{align}
 G_{k,D}(z)=&\frac{L_{D}(1-k)}{2}+\sum_{n=1}^{\infty}\left(\sum_{d\mid n}\legendre{D}{d}d^{k-1}\right)q^n.
 \end{align}
Hence for $n\ge1$ the $n$-th Fourier coefficient of $U_2G_{k,D}(z)$ is
\begin{align}
    \sum_{d\mid 2n}\legendre{D}{d}d^{k-1}.
\end{align}
On the other hand, the $n$-th Fourier coefficient of right hand side of \eqref{eq:U2formular} is (using the convention that $d\mid(n/2)$ is an empty sum when $2\nmid n$)
    \begin{align}
       &\sum_{d\mid n}\legendre{D}{d}d^{k-1}+\sum_{d\mid n}\legendre{D}{2d}(2d)^{k-1}-\sum_{d\mid\frac{n}{2}}\legendre{D}{2d}(2d)^{k-1}\\
       =&\sum_{\substack{d\mid n\\d~{\rm even}}}\legendre{D}{d}d^{k-1}+\sum_{\substack{d\mid n\\d~{\rm odd}}}\legendre{D}{d}d^{k-1}+\sum_{\substack{m\mid 2n\\m~{\rm even}}}\legendre{D}{m}m^{k-1}-\sum_{\substack{\ell\mid n\\ \ell~{\rm even}}}\legendre{D}{\ell}\ell^{k-1}\\=&\sum_{\substack{d\mid n\\d~{\rm even}}}\legendre{D}{d}d^{k-1}+\sum_{\substack{d\mid 2n\\d~{\rm odd}}}\legendre{D}{d}d^{k-1}+\sum_{\substack{m\mid 2n\\m~{\rm even}}}\legendre{D}{m}m^{k-1}-\sum_{\substack{\ell\mid n\\ \ell~{\rm even}}}\legendre{D}{\ell}\ell^{k-1}
       \\=&\sum_{m\mid 2n}\legendre{D}{m}m^{k-1},
    \end{align}
which gives the result.
\end{proof}
\begin{lemma}\label{lem:gsimplified}The following identity holds:
\[G_{k,D}(4z) -G_{k,D}(8z)-2^{-k}\left(\frac{D}{2}\right) \left(G_{k,D}\left(2z+\frac{1}{2}\right) + G_{k,D}(2z)\right)=-\left(\frac{D}{2}\right)2^{-k+1}G_{k,D}(4z).\]
\end{lemma}
\begin{proof}
From Lemma \ref{lem:U2ofGkd}, we know that
\[        \frac{1}{2}\left(G_{k,D}\left(\frac{z}{2}\right) + G_{k,D}\left(\frac{z+1}{2}\right) \right)=\left(1+2^{k-1}\legendre{D}{2}\right)G_{k,D}(z)-2^{k-1}\legendre{D}{2}G_{k,D}(2z).\]
Performing a change of variable with $z \mapsto 4z$, we get 
\[        \frac{1}{2}\left(G_{k,D}\left(2z\right) + G_{k,D}\left(2z+\frac{1}{2}\right) \right)=\left(1+2^{k-1}\legendre{D}{2}\right)G_{k,D}(4z)-2^{k-1}\legendre{D}{2}G_{k,D}(8z),\]
which implies that
\begin{align}
&\hspace{15pt}G_{k,D}(4z) -G_{k,D}(8z)-2^{-k}\left(\frac{D}{2}\right) \left(G_{k,D}\left(2z+\frac{1}{2}\right) + G_{k,D}(2z)\right)\\
&=G_{k,D}(4z) -G_{k,D}(8z)-2^{-k+1}\left(\frac{D}{2}\right) \left(\left(1+2^{k-1}\legendre{D}{2}\right)G_{k,D}(4z)-2^{k-1}\legendre{D}{2}G_{k,D}(8z)\right)\\
&= G_{k,D}(4z)-G_{k,D}(8z)-2^{-k+1}\legendre{D}{2}G_{k,D}(4z)-G_{k,D}(4z)-G_{k,D}(8z)\\
&=-\legendre{D}{2}2^{-k+1}G_{k,D}(4z),
\end{align}
as desired. 
\end{proof}Note that $\gamma_v=\begin{bsmallmatrix}1 & 0\\4|D|v & 1\end{bsmallmatrix}$ for $v=0,1,2,3$ form a system of representatives of $\Gamma_0(16|D|)\backslash\Gamma_0(4|D|)$ \cite[p.~195]{Kohnen-Zagier1981}. The following Lemma explicitly computes each term in $\Tr^{16D}_{4D}(VG_{k,D}(2z))$. 
\begin{lemma}\label{lem:gvs}
For $\gamma_v=\begin{bsmallmatrix}1 & 0\\4|D|v & 1\end{bsmallmatrix}$, we have
\begin{align}
V(G_{k,D}(2z))\Big|_{k}\gamma_v&=\begin{cases}
 G_{k,D}\left(2z+\frac{1}{2}\right) & v\equiv0,2\pmod4 ,\\ \legendre{D}{2}2^{k}G_{k,D}(8z) &   v\equiv1,3\pmod4.
\end{cases}
\end{align}

\end{lemma}
\begin{proof}
First,
\begin{align}
V(G_{k,D}(2z))\Bigg|_{k}\begin{bmatrix}1 & 0\\
4|D|v & 1
\end{bmatrix}&=2^{-k/2}G_{k,D}(z)\Bigg|_{k}\begin{bmatrix}2 & 0\\
0 & 1
\end{bmatrix}\begin{bmatrix}4 & 1\\
0 & 4
\end{bmatrix}\begin{bmatrix}1 & 0\\
4|D|v & 1
\end{bmatrix}\\&=2^{-k/2}G_{k,D}(z)\Bigg|_{k}\begin{bmatrix}8(|D|v+1) & 2\\
16|D|v & 4
\end{bmatrix}.\end{align}
Now we do some casework. 
\begin{enumerate}
\item $v=0$: We have
\[V(G_{k,D}(2z))\Bigg|_{k}\begin{bmatrix}1 & 0\\
0 & 1
\end{bmatrix}=V(G_{k,D}(2z))=G_{k,D}\left(2\left(z+\frac14\right)\right)=G_{k,D}\left(2z+\frac{1}{2}\right).\]
\item $v=1, 3$: Since $v$ and $|D|$ are odd, $v|D|+1$ must be even, $\gcd(\frac{|D|v+1}{2}, |D|v)=1$, and there exist some $x, y \in \mathbb{Z}$ such that $\frac{|D|v+1}{2}x+|D|vy=1$. Note also that $x\equiv2\pmod D$ and $\legendre{D}{x}=\legendre{D}{2}$. Thus 
\begin{align}
V(G_{k,D}(2z))\Bigg|_{k}\begin{bmatrix}1 & 0\\
4|D|v & 1
\end{bmatrix}&=2^{-k/2}G_{k,D}(z)\Bigg|_{k}\begin{bmatrix}8(|D|v+1) & 2\\
16|D|v & 4
\end{bmatrix}\\
&=2^{-k/2}G_{k,D}(z)\Bigg|_{k}\hspace{-3pt}\begin{bmatrix}\frac{|D|v+1}{2} & -y\\
|D|v & x
\end{bmatrix}\hspace{-5pt}\begin{bmatrix}16 & 2x+4y\\
0 & 2
\end{bmatrix}\\
&=2^{-k/2}\left(\frac{D}{x}\right)G_{k,D}(z)\Bigg|_{k}\begin{bmatrix}16 & 2x+4y\\
0 & 2
\end{bmatrix}\\
&=2^{k}\left(\frac{D}{x}\right)G_{k,D}(8z+x+2y)\\&=\left(\frac{D}{2}\right)2^{k}G_{k,D}(8z).
\end{align}
\item $v=2$: Since $\text{gcd}(2|D|+1, 4|D|)=1$, we can pick $x,y\in\mathbb{Z}$ such that $ (2|D|+1)x+4|D|y=1$. As $4|D|y$ is even, $x$ must be odd, so $G_{k,D}(2z+\frac{x}{2})=G_{k,D}(2z+\frac{1}{2})$, and further $\left(\frac{D}{x}\right)=1$ since $x\equiv1\pmod D$. Hence
\begin{align}
V(G_{k,D}(2z))\Bigg|_{k}\begin{bmatrix}1 & 0\\
8|D| & 1
\end{bmatrix}&=2^{-k/2}G_{k,D}(z)\Bigg|_{k}\begin{bmatrix}8(2|D|+1) & 2\\
32|D| & 4
\end{bmatrix}\\&=2^{-k/2}G_{k,D}(z)\Bigg|_{k}\begin{bmatrix}2|D|+1 & -y\\
4|D| & x
\end{bmatrix}\begin{bmatrix}8 & 2x+4y\\
0 & 4
\end{bmatrix}\\
&=\left(\frac{D}{x}\right)2^{-k/2}G_{k,D}(z)\Bigg|_{k}\begin{bmatrix}8 & 2x+4y\\
0 & 4
\end{bmatrix}\\&=G_{k,D}\left(2z+\frac{x}{2}\right)\\&=G_{k,D}\left(2z+\frac{1}{2}\right).
\end{align}
\end{enumerate}
Thus, the proof is complete.

\end{proof}
The following Lemma explicitly computes each term in $\Tr^{16D}_{4D}(V\theta(|D|z))$. 
\begin{lemma}\label{lem:thetavs}
Let $D$ be an odd fundamental discriminant. For $1\leq v\leq4$, we define $\gamma_v=\begin{bsmallmatrix}1 & 0\\4|D|v & 1\end{bsmallmatrix}$. Then we have

\begin{align}
V(\theta(|D|z))\Big|_{\frac{1}{2}}\gamma_0&=\theta\left(|D|z+\frac{|D|}{4}\right),\\
V(\theta(|D|z))\Big|_{\frac{1}{2}}\gamma_1&=\begin{cases}(2i)^{1/2}(\theta(|D|z)-\theta(4|D|z)) & D>0,\\
-i(2i)^{1/2}\theta(4|D|z) &D<0,
\end{cases}
\\
V(\theta(|D|z))\Big|_{\frac{1}{2}}\gamma_2&={\rm sgn}(D)i\theta\left(|D|z-\frac{|D|}{4}\right),\\
V(\theta(|D|z))\Big|_{\frac{1}{2}}\gamma_3&=\begin{cases}(2i)^{1/2}\theta(4|D|z) & D>0,\\
    -i(2i)^{1/2}(\theta(|D|z)-\theta(4|D|z)) & D<0,
\end{cases}
\end{align}
taking the principal branch of every square root. 

\end{lemma}
\begin{proof}Note that
\begin{align}
V(\theta(|D|z))\Bigg|_{\frac{1}{2}}\begin{bmatrix}1 & 0\\
4|D|v & 1
\end{bmatrix}&=|D|^{-1/4}\theta(z)\Bigg|_{\frac{1}{2}}\begin{bmatrix}|D| & 0\\
0 & 1
\end{bmatrix}\begin{bmatrix}4 & 1\\
0 & 4
\end{bmatrix}\begin{bmatrix}1 & 0\\
4|D|v & 1
\end{bmatrix}\\
&=|D|^{-1/4}\theta(z)\Bigg|_{\frac{1}{2}}\begin{bmatrix}|D|v+1 & |D|\\
4v & 4
\end{bmatrix}\begin{bmatrix}4|D| & 0\\
0 & 1\\
\end{bmatrix}.
\end{align}
Now we do some more casework. 
\begin{enumerate}
\item $v=0$: 
\[V(\theta(|D|z))\Bigg|_{\frac{1}{2}}\begin{bmatrix}1 & 0\\
0 & 1
\end{bmatrix}=\theta\left(|D|z+\frac{|D|}{4}\right).\]
\item $v=1$: We have

\begin{align}
V(\theta(|D|z))\Bigg|_{\frac{1}{2}}\begin{bmatrix}1 & 0\\
4|D| & 1
\end{bmatrix}&=|D|^{-1/4}\theta(z)\Bigg|_{\frac{1}{2}}\begin{bmatrix}|D|+1 & |D|\\
4 & 4
\end{bmatrix}\begin{bmatrix}4|D| & 0\\
0 & 1
\end{bmatrix}\\
&=|D|^{-1/4}\theta(z)\Bigg|_{\frac{1}{2}}\begin{bmatrix}-|D| & |D|+1\\
-4 & 4
\end{bmatrix}\begin{bmatrix}0 & -1\\
4 & 0
\end{bmatrix}\begin{bmatrix}|D| & 0\\
0 & 1
\end{bmatrix}.
\end{align}
Since $|D|$ is odd, we can choose  $x,y\in\mathbb{Z}$ such that $-|D|x-4y=1$. This gives us
\begin{align}
\begin{bmatrix}-|D| & |D|+1\\
-4 & 4
\end{bmatrix}
=\begin{bmatrix}-|D| & -y\\
-4 & x
\end{bmatrix}\begin{bmatrix}1 & x-1\\
0 & 4
\end{bmatrix}.
\end{align}
Note that $\begin{bmatrix}-|D| & -y\\
-4 & x
\end{bmatrix}\in \Gamma_0(4)$, so we have
\begin{align}
V(\theta(|D|z))\Bigg|_{\frac{1}{2}}\begin{bmatrix}1 & 0\\
4|D| & 1
\end{bmatrix}&=|D|^{-1/4}\left(\frac{-4}{x}\right)\varepsilon_{x}^{-1}\theta(z)\Bigg|_{\frac{1}{2}}\begin{bmatrix}1 & x-1\\
0 & 4
\end{bmatrix}\begin{bmatrix}0 & -1\\
4 & 0
\end{bmatrix}\begin{bmatrix}|D| & 0\\
0 & 1
\end{bmatrix}\\
&=2^{-1/2}|D|^{-1/4}\left(\frac{-4}{x}\right)\varepsilon_{x}^{-1}\theta\left(\frac{z+x-1}{4}\right)\Bigg|_{\frac{1}{2}}\begin{bmatrix}0 & -1\\
4 & 0
\end{bmatrix}\begin{bmatrix}|D| & 0\\
0 & 1
\end{bmatrix},
\end{align}
where 
$$\varepsilon_x=\begin{cases}
    1 &x\equiv 1\pmod4,\\ i &x\equiv 3\pmod 4.
\end{cases}$$
Now we have two cases since $-|D|x-4y=1$ and the sign of $D$ determines $\varepsilon_x$.
\begin{enumerate} 
\item If $D>0$, then $|D|\equiv 1\pmod 4$, so $x\equiv3\pmod 4$, $\theta\left(\frac{z+x-1}{4}\right)=\theta\left(\frac{z}{4}+\frac{1}{2}\right)=2\theta(z)-\theta(\frac{z}{4})$, $\varepsilon_{x}=i$, and $\left(\frac{-4}{x}\right)=-1$. So we have 
\begin{align}
V(\theta(|D|z))\Bigg|_{\frac{1}{2}}\begin{bmatrix}1 & 0\\
4|D| & 1
\end{bmatrix}&=2^{-1/2}|D|^{-1/4}\left(\frac{-4}{x}\right)\varepsilon_{x}^{-1}\theta\left(\frac{z+x-1}{4}\right)\Bigg|_{\frac{1}{2}}\begin{bmatrix}0 & -1\\
4 & 0
\end{bmatrix}\begin{bmatrix}|D| & 0\\
0 & 1
\end{bmatrix}\\
&=i2^{-1/2}|D|^{-1/4}\left(2\theta(z)-\theta\left(\frac{z}{4}\right)\right)\Bigg|_{\frac{1}{2}}\begin{bmatrix}0 & -1\\
4 & 0
\end{bmatrix}\begin{bmatrix}|D| & 0\\
0 & 1
\end{bmatrix}.
\end{align}
Explicitly computing these, we get
\begin{align}
\theta(z)\Bigg|_{\frac{1}{2}}\begin{bmatrix}0 & -1\\
4 & 0
\end{bmatrix}\begin{bmatrix}|D| & 0\\
0 & 1
\end{bmatrix}&=i^{-1/2}|D|^{1/4}\theta(|D|z),\label{eq:transformationthetaz}\\
\theta\left(\frac{z}{4}\right)\Bigg|_{\frac{1}{2}}\begin{bmatrix}0 & -1\\
4 & 0
\end{bmatrix}\begin{bmatrix}|D| & 0\\
0 & 1
\end{bmatrix}&=2^{1/2}\theta(z)\Bigg|_{\frac{1}{2}}\begin{bmatrix}1 & 0\\
0 & 4
\end{bmatrix}\begin{bmatrix}0 & -1\\
4 & 0
\end{bmatrix}\begin{bmatrix}|D| & 0\\
0 & 1\\
\end{bmatrix}\\
&=2^{1/2}\theta(z)\Bigg|_{\frac{1}{2}}\begin{bmatrix}0 & -1\\
4 & 0
\end{bmatrix}\begin{bmatrix}4 & 0\\
0 & 1
\end{bmatrix}\begin{bmatrix}|D| & 0\\
0 & 1
\end{bmatrix}\\
&=2^{1/2}i^{-1/2}\theta(z)\Bigg|_{\frac{1}{2}}\begin{bmatrix}4 & 0\\
0 & 1
\end{bmatrix}\begin{bmatrix}|D| & 0\\
0 & 1
\end{bmatrix}\\
&=2i^{-1/2}|D|^{1/4}\theta(4|D|z). \label{eq:transformationthetaz/4}
\end{align}
So our expression simplifies to 
\begin{align}
i2^{-1/2}|D|^{-1/4}\left(2\theta(z)-\theta\left(\frac{z}{4}\right)\right)\Bigg|_{\frac{1}{2}}\begin{bmatrix}0 & -1\\
4 & 0
\end{bmatrix}\begin{bmatrix}|D| & 0\\
0 & 1
\end{bmatrix}=(2i)^{1/2}(\theta(|D|z)-\theta(4|D|z)).
\end{align}
\item If $D<0$, then $|D|\equiv3\pmod 4$, $x\equiv1\pmod 4$, $\theta\left(\frac{z+x-1}{4}\right)=\theta(\frac{z}{4}),\varepsilon_{x}=1,$ and $\left(\frac{-4}{x}\right)=1$. So we have 
\begin{align}
V(\theta(|D|z))\Bigg|_{\frac{1}{2}}\begin{bmatrix}1 & 0\\
4|D| & 1
\end{bmatrix}&=2^{-1/2}|D|^{-1/4}\left(\frac{-4}{x}\right)\varepsilon_{x}^{-1}\theta\left(\frac{z+x-1}{4}\right)\Bigg|_{\frac{1}{2}}\begin{bmatrix}0 & -1\\
4 & 0
\end{bmatrix}\begin{bmatrix}|D| & 0\\
0 & 1
\end{bmatrix}\\
&=2^{-1/2}|D|^{-1/4}\theta\left(\frac{z}{4}\right)\Bigg|_{\frac{1}{2}}\begin{bmatrix}0 & -1\\
4 & 0
\end{bmatrix}\begin{bmatrix}|D| & 0\\
0 & 1
\end{bmatrix}\\
&=-i(2i)^{1/2}\theta(4|D|z).
\end{align}
\end{enumerate}

\item $v=2$: Since $2|D|+1$ is coprime to $8$, we can find $x,y\in\mathbb{Z}$ such that $(2|D|+1)x+8y=1$. 
\begin{align}
V(\theta(|D|z))\Bigg|_{\frac{1}{2}}\begin{bmatrix}1 & 0\\
8|D| & 1
\end{bmatrix}&=|D|^{-1/4}\theta(z)\Bigg|_{\frac{1}{2}}\begin{bmatrix}2|D|+1 & |D|\\
8 & 4
\end{bmatrix}\begin{bmatrix}4|D| & 0\\
0 & 1
\end{bmatrix}\\
&=|D|^{-1/4}\theta(z)\Bigg|_{\frac{1}{2}}\begin{bmatrix}2|D|+1 & -y\\
8 & x
\end{bmatrix}\begin{bmatrix}1 & |D|x+4y\\
0 & 4
\end{bmatrix}\begin{bmatrix}4|D| & 0\\
0 & 1
\end{bmatrix}.
\end{align}
Now that $\begin{bmatrix}2|D|+1 & -y\\
8 & x \end{bmatrix}$ is in $\Gamma_{0}(4)$, we get
\begin{align}
V(\theta(|D|z))\Bigg|_{\frac{1}{2}}\begin{bmatrix}1 & 0\\
8|D| & 1
\end{bmatrix}&=|D|^{-1/4}\varepsilon_{x}^{-1}\left(\frac{8}{x}\right)\theta(z)\Bigg|_{\frac{1}{2}}\begin{bmatrix}1 & \frac{1-x}{2}\\
0 & 4
\end{bmatrix}\begin{bmatrix}4|D| & 0\\
0 & 1
\end{bmatrix}\\
&=|D|^{-1/4}\varepsilon_{x}^{-1}\left(\frac{8}{x}\right)2^{-1/2}\theta\left(\frac{z}{4}+\frac{1-x}{8}\right)\Bigg|_{\frac{1}{2}}\begin{bmatrix}4|D| & 0\\
0 & 1
\end{bmatrix}\\
&=\varepsilon_{x}^{-1}\left(\frac{8}{x}\right)\theta\left(|D|z+\frac{1-x}{8}\right)
\end{align}
As $(2|D|+1)x+8y=1$ and the sign of $D$ determines $\varepsilon_x$, we do casework again. 
\begin{enumerate} 
\item If $D>0$, then $|D|\equiv1\pmod 4$, which implies that $3x\equiv1\pmod 8$, $x\equiv3\pmod 8$, $\varepsilon_x=i$ and $\legendre{8}{x}=\legendre{8}{3}=-1$. Note also that $\theta(|D|z+\frac{1-x}{8})=\theta(|D|z-\frac{1}{4})$. 
\item If $D<0$, then $|D|\equiv3\pmod 4$, which gives $7x\equiv1\pmod 8, x\equiv7\pmod 8$, $\legendre{8}{x}=\legendre{8}{7}=1$, $\varepsilon_x=i$ and $\theta(|D|z+\frac{1-x}{8})=\theta(|D|z-\frac{3}{4})$. 
\end{enumerate}
Combining these two cases, we can write  
\[V(\theta(|D|z))\Bigg|_{\frac{1}{2}}\begin{bmatrix}1 & 0\\
4|D|v & 1
\end{bmatrix}=\text{sgn}(D)i\theta\left(|D|z-\frac{|D|}{4}\right).\]
\item $v=3$: Note that
\begin{align}
V(\theta(|D|z))\Bigg|_{\frac{1}{2}}\begin{bmatrix}1 & 0\\
4|D| & 1
\end{bmatrix}&=|D|^{-1/4}\theta(z)\Bigg|_{\frac{1}{2}}\begin{bmatrix}3|D|+1 & |D|\\
12 & 4
\end{bmatrix}\begin{bmatrix}4|D| & 0\\
0 & 1
\end{bmatrix}\\
&=|D|^{-1/4}\theta(z)\Bigg|_{\frac{1}{2}}\begin{bmatrix}-|D| & 3|D|+1\\
-4 & 12
\end{bmatrix}\begin{bmatrix}0 & -1\\
4 & 0
\end{bmatrix}\begin{bmatrix}|D| & 0\\
0 & 1
\end{bmatrix}.
\end{align}
Since $\gcd(4,D)=1$, we can pick $x,y\in\mathbb{Z}$ such that $-|D|x-4y=1$. This gives us
\begin{align}
\begin{bmatrix}-|D| & 3|D|+1\\
-4 & 12
\end{bmatrix}
\begin{bmatrix}0 & -1\\
4 & 0
\end{bmatrix}\begin{bmatrix}|D| & 0\\
0 & 1
\end{bmatrix}&
=\begin{bmatrix}-|D| & -y\\
-4 & x
\end{bmatrix}\begin{bmatrix}1 & x-3\\
0 & 4
\end{bmatrix}\begin{bmatrix}0 & -1\\
4 & 0
\end{bmatrix}\begin{bmatrix}|D| & 0\\
0 & 1
\end{bmatrix}.
\end{align}
Note that $\begin{bmatrix}-|D| & -y\\
-4 & x
\end{bmatrix}$ is in $\Gamma_{0}(4)$, so we have
\begin{align}
V(\theta(|D|z))\Bigg|_{\frac{1}{2}}\begin{bmatrix}1 & 0\\
4|D| & 1
\end{bmatrix}&=|D|^{-1/4}\left(\frac{-4}{x}\right)\varepsilon_{x}^{-1}\theta(z)\Bigg|_{\frac{1}{2}}\begin{bmatrix}1 & x-3\\
0 & 4
\end{bmatrix}\begin{bmatrix}0 & -1\\
4 & 0
\end{bmatrix}\begin{bmatrix}|D| & 0\\
0 & 1
\end{bmatrix}\\
&=2^{-1/2}|D|^{-1/4}\left(\frac{-4}{x}\right)\varepsilon_{x}^{-1}\theta\left(\frac{z+x-3}{4}\right)\Bigg|_{\frac{1}{2}}\begin{bmatrix}0 & -1\\
4 & 0
\end{bmatrix}\begin{bmatrix}|D| & 0\\
0 & 1
\end{bmatrix}.
\end{align}
Again, we have two cases.
\begin{enumerate}  
\item If $D>0$, then $|D|=1\pmod 4$, $x\equiv3\pmod 4$, $\theta\left(\frac{z+x-3}{4}\right)=\theta\left(\frac{z}{4}\right)$ and $\varepsilon_{x}=i,\left(\frac{-4}{x}\right)=-1$. Applying \eqref{eq:transformationthetaz/4} to above equation, we get
\begin{align}
V(\theta(|D|z))\Bigg|_{\frac{1}{2}}\begin{bmatrix}1 & 0\\
4|D| & 1
\end{bmatrix}&=2^{-1/2}|D|^{-1/4}\left(\frac{-4}{x}\right)\varepsilon_{x}^{-1}\theta\left(\frac{z+x-3}{4}\right)\Bigg|_{\frac{1}{2}}\begin{bmatrix}0 & -1\\
4 & 0
\end{bmatrix}\begin{bmatrix}|D| & 0\\
0 & 1
\end{bmatrix}\\
&=i2^{-1/2}|D|^{-1/4}\theta\left(\frac{z}{4}\right)\Bigg|_{\frac{1}{2}}\begin{bmatrix}0 & -1\\
4 & 0
\end{bmatrix}\begin{bmatrix}|D| & 0\\
0 & 1
\end{bmatrix}\\
&=(2i)^{1/2}\theta(4|D|z).
\end{align}
\item If $D<0$, then $|D|=3\pmod 4$, $x\equiv1\pmod 4$, $\theta\left(\frac{z+x-3}{4}\right)=\theta\left(\frac{z}{4}+\frac{1}{2}\right)=2\theta(z)-\theta(\frac{z}{4})$, $\varepsilon_{x}=1$, and $\left(\frac{-4}{x}\right)=1$. So we get 
\begin{align}
V(\theta(|D|z))\Bigg|_{\frac{1}{2}}\begin{bmatrix}1 & 0\\
12|D| & 1
\end{bmatrix}&=2^{-1/2}|D|^{-1/4}\left(2\theta(z)-\theta\left(\frac{z}{4}\right)\right)\Bigg|_{\frac{1}{2}}\begin{bmatrix}0 & -1\\
4 & 0
\end{bmatrix}\begin{bmatrix}|D| & 0\\
0 & 1
\end{bmatrix}\\
&=-i(2i)^{1/2}(\theta(|D|z)-\theta(4|D|z)),
\end{align}
where the last equality comes from \eqref{eq:transformationthetaz} and \eqref{eq:transformationthetaz/4}.
\end{enumerate}
\end{enumerate}
This finishes the proof.
\end{proof}
We also need the following two technical lemmas.
\begin{lemma}\label{lem:thetaodds}We have that 
\[V\theta(|D|z)\Big|_{\frac{1}{2}}\gamma_1+V\theta(|D|z)\Big|_{\frac{1}{2}}\gamma_3 = \varepsilon_{|D|}^{-1}(2i)^{1/2}\theta(|D|z).\]
\end{lemma}
\begin{proof}
It is a trivial consequence of Lemma \ref{lem:thetavs}. 
\end{proof}
\begin{lemma}\label{lem:thetaevens}We have that
\[V\theta(|D|z)\Big|_{\frac{1}{2}}\gamma_0+V\theta(|D|z)\Big|_{\frac{1}{2}}\gamma_2=(1+i{\rm sgn}(D))\theta(|D|z).\]
\end{lemma}
\begin{proof}By Lemma  \ref{lem:thetavs}, we have
\begin{align}V\theta(|D|z)\Big|_{\frac{1}{2}}\gamma_0+V\theta(|D|z)\Big|_{\frac{1}{2}}\gamma_2=\theta\left(|D|z+\frac{|D|}{4}\right)+i\text{sgn}(D)\theta\left(|D|z-\frac{|D|}{4}\right)\end{align}
Note that 
\[\theta\left(|D|z+\frac{|D|}{4}\right) = \sum\limits_{n\in\mathbb{Z}}e^{2\pi i \frac{n^2|D|}{4}}e^{2\pi i n^2|D|z}=\sum\limits_{n\in\mathbb{Z}}a(n)e^{2\pi i n^2|D|z}\]
where $a(n)=i\text{sgn}(D)$ if $n$ is odd and $a(n)=1$ if $n$ is even. On the other hand, 
\[\theta\left(|D|z-\frac{|D|}{4}\right) = \sum\limits_{n\in\mathbb{Z}}e^{2\pi i \frac{-n^2|D|}{4}}e^{2\pi i n^2|D|z}=\sum\limits_{n\in\mathbb{Z}}b(n)e^{2\pi i n^2|D|z}\]
where $b(n)=-i\text{sgn}(D)$ if $n$ is odd and $b(n)=1$ if $n$ is even. Hence 
\begin{align}
\theta\left(|D|z+\frac{|D|}{4}\right)+{\rm sgn}(D)i\theta\left(|D|z-\frac{|D|}{4}\right)&=
\sum\limits_{n\in\mathbb{Z}}(a(n)+i\text{sgn}(D)(b(n))e^{2\pi i n^2|D|z}\\
&= (1+i\text{sgn}(D))\theta(|D|z),
\end{align}
as desired.
\end{proof}
Now we are ready to prove the alternate formula for $\mathcal{G}_{D,k,e}$ \eqref{eq:G_D,k,e} promised at the beginning of this section.
\begin{proposition}\label{prop:projectioncomputation}
Let $k\ge 4$ and $e>0$ be integers such that $k+2e=\ell$ and let $D$ be an odd fundamental discriminant such that $(-1)^\ell D>0$.
Then  
\begin{equation} \label{eq:gkdepr}
\mathcal{G}_{D, k, e}(z)=\Tr^{4D}_4 [G_{k,D}(4z), \theta(|D|z)]_e
\end{equation}
\end{proposition}
\begin{proof}
We closely follow \cite[p.~195]{Kohnen-Zagier1981}, where a similar result is implicit in the proof of formulas \cite[(6), (7)]{Kohnen-Zagier1981}. 
Write $h=[G_{k,4D}(z),\theta(|D|z)]_e$. By Lemma \ref{lem:vswap} we get
\begin{align}
\mathcal{G}_{D,k,e}&=\frac32\left(1-\legendre{D}{2}2^{-k}\right)^{-1}\pr^+\Tr_4^{4D}(h)\\
&=\frac32\left(1-\legendre{D}{2}2^{-k}\right)^{-1}\left(\frac{1-(-1)^\ell i}{6}\Tr_4^{16}V(\Tr^{4D}_{4}(h)) + \frac{1}{3}\Tr^{4D}_{4}(h)\right)\\
&=\frac32\left(1-\legendre{D}{2}2^{-k}\right)^{-1}\Tr_4^{4D}\left( \frac{1-(-1)^\ell i}{6}\Tr^{16D}_{4D}(V(h)) + \frac{1}{3}h\right)\\
&=\frac32\left(1-\legendre{D}{2}2^{-k}\right)^{-1}\Tr_4^{4D}g_D,\label{eq:prcalculation}
\end{align}
with
$$g_D =  \frac{1-(-1)^k i}{6}\Tr^{16D}_{4D}(V(h)) + \frac{1}{3}h.$$ 
Note that $k\equiv \ell$ mod 2, so we can substitute in $(-1)^k$ for $(-1)^\ell$ above. We now compute $g_D$. The matrices $\gamma_v = \begin{bsmallmatrix}1 & 0\\4|D|v & 1\end{bsmallmatrix}$, where $v=0,1,2,3$, form a set of coset representatives for $\Gamma_0(16|D|)\backslash\Gamma_0(4|D|)$ \cite[p.~195]{Kohnen-Zagier1981}. Then we have 
\begin{align}
\Tr^{16D}_{4D}(V(h))&=\Tr^{16D}_{4D}(V[G_{k,4D}(z),\theta(|D|z)]_e)\\
&=\Tr^{16D}_{4D}\left[VG_{k,4D}(z),V\theta(|D|z)\right]_e\\
&=\Tr^{16D}_{4D}\left[G_{k,D}(4z)-2^{-k}\left(\frac{D}{2}\right)V(G_{k,D}(2z)),V\theta(|D|z)\right]_e\\
&=\sum\limits_{\gamma_v}\left[G_{k,D}(4z)-2^{-k}\left(\frac{D}{2}\right)V(G_{k,D}(2z)),V\theta(|D|z)\right]_e\Bigg|_{k+\frac{1}{2}+2e}\gamma_v\\
&=\sum\limits_{\gamma_v}\left[G_{k,D}(4z)\Big|_{k}\gamma_v-2^{-k}\left(\frac{D}{2}\right)V(G_{k,D}(2z))\Big|_{k}\gamma_v,V\theta(|D|z)\Big|_{\frac{1}{2}}\gamma_v\right]_e.
\end{align}
Since $\gamma_v\in \Gamma_0(4|D|)$, $G_{k,D}(4z)\Big|_{k}\gamma_v=G_{k,D}(4z)$. By Lemma \ref{lem:gvs}, we get 
\begin{align}
\Tr^{16D}_{4D}(V(h))&=\sum\limits_{\gamma_v}\left[G_{k,D}(4z)\Big|_{k}\gamma_v-2^{-k}\left(\frac{D}{2}\right)V(G_{k,D}(2z))\Big|_{k}\gamma_v,V\theta(|D|z)\Big|_{\frac{1}{2}}\gamma_v\right]_e\\
&=\left[G_{k,D}(4z)-2^{-k}\left(\frac{D}{2}\right)G_{k,D}\left(2z+\frac{1}{2}\right),V\theta(|D|z)\Big|_{\frac{1}{2}}\gamma_0+V\theta(|D|z)\Big|_{\frac{1}{2}}\gamma_2\right]_e\\
&\hspace{20pt}+\left[G_{k,D}(4z)- G_{k,D}(8z),V\theta(|D|z)\Big|_{\frac{1}{2}}\gamma_1+V\theta(|D|z)\Big|_{\frac{1}{2}}\gamma_3\right]_e.
\end{align}
Using Lemmas \ref{lem:thetaodds} and \ref{lem:thetaevens} and noting that sgn$(D)=(-1)^k$ by our assumption, we can simplify this to
\begin{align}
\Tr^{16D}_{4D}(V(h))
&=\left[G_{k,D}(4z)-2^{-k}\left(\frac{D}{2}\right)G_{k,D}\left(2z+\frac{1}{2}\right),(1-i(-1)^k)\theta(|D|z)\right]_e\\
&\hspace{20pt}+\left[G_{k,D}(4z)- G_{k,D}(8z),\varepsilon_{|D|}^{-1}(2i)^{1/2}\theta(|D|z)\right]_e.
\end{align}
Now we can finally compute the projection. 
\begin{align}
g_D(z)&=  \frac{1-(-1)^k i}{6}\Tr^{16D}_{4D}(V(h(z))) + \frac{1}{3}h(z)\\
&= \frac{1-i(-1)^k}{6}\left[G_{k,D}(4z)-2^{-k}\left(\frac{D}{2}\right)G_{k,D}\left(2z+\frac{1}{2}\right),(1+i(-1)^k)\theta(|D|z)\right]_e\\
&\hspace{20pt}+ \frac{1-i(-1)^k}{6}\left[G_{k,D}(4z)- G_{k,D}(8z),\varepsilon_{|D|}^{-1}(2i)^{1/2}\theta(|D|z)\right]_e\\
&\hspace{20pt}+ \frac{1}{3}\left[G_{k,D}(4z)-2^{-k}\left(\frac{D}{2}\right)G_{k,D}(2z),\theta(|D|z)\right]_e\\
&=\frac{1}{3}\left[G_{k,D}(4z)-2^{-k}\left(\frac{D}{2}\right)G_{k,D}\left(2z+\frac{1}{2}\right),\theta(|D|z)\right]_e\\
&\hspace{20pt}+\frac{1}{3}\left[G_{k,D}(4z)- G_{k,D}(8z),\theta(|D|z)\right]_e\\
&\hspace{20pt}+ \frac{1}{3}\left[G_{k,D}(4z)-2^{-k}\left(\frac{D}{2}\right)G_{k,D}(2z),\theta(|D|z)\right]_e\\
&=\frac{1}{3}\left[G_{k,D}(4z) -G_{k,D}(8z)-2^{-k}\left(\frac{D}{2}\right) \left(G_{k,D}\left(2z+\frac{1}{2}\right) + G_{k,D}(2z)\right), \theta(|D|z)\right]_e\\
&\hspace{20pt}+\frac{2}{3}[G_{k,D}(4z),\theta(|D|z)]_e.
\end{align}
Using Lemma \ref{lem:gsimplified}, we can simplify the first term to get 
\begin{align}
g_D(z)&=\frac{1}{3}\left[-\left(\frac{D}{2}\right)2^{-k+1}G_{k,D}(4z), \theta(|D|z)\right]_e+\frac{2}{3}[G_{k,D}(4z),\theta(|D|z)]_e\\
&=\frac{2}{3}\left(1-\left(\frac{D}{2}\right)2^{-k}\right)\left[G_{k,D}(4z), \theta(|D|z)\right]_e.\label{eq:gDfinal}
\end{align}
Plugging \eqref{eq:gDfinal} into \eqref{eq:prcalculation} gives the desired result.
\end{proof}

\section{Eisenstein Series}\label{sec:Eisenstein}
In this section, we define various Eisenstein series and show that $G_{k,4D}(z)$ \eqref{eq:G_k,4D} is an Eisenstein series for the cusp at infinity of level $4|D|$. 
We recall the theory of Eisenstein series as developed in Miyake \cite[\S7]{MiyakeMFbook}. Let $\chi$ and $\psi$ be Dirichlet characters mod $L$ and mod $M$, respectively. For $k\geq3$, we put
\begin{align}
    E_{k}(z;\chi,\psi)=\sideset{}{'}\sum\limits_{m,n\in\mathbb{Z}}\chi(m)\psi(n)(mz+n)^{-k}.
\end{align}
Here  $\sump$ is the summation over all pairs of integers $(m,n)$ except $(0,0)$. For a fundamental discriminant $D$, we write   $\chi_D(\cdot)=\legendre{D}{\cdot}$ and $L_D(k)=\sum_{n=1}^{\infty}\chi_D(n)n^{-k}$.
\begin{lemma}[{\cite[Theorem 7.1.3]{MiyakeMFbook}}] \label{lem:EisensteinFourierexpansion}
    Assume $k\geq3$. Let $\chi$ and $\psi$ be Dirichlet characters mod $L$ and mod $M$, respectively, satisfying $\chi(-1)\psi(-1)=(-1)^k$. Let $m_{\psi}$ be the conductor of $\psi$, and $\psi^0$ be the primitive character associated with $\psi$. Then 
    \begin{align}
        E_{k}(z;\chi,\psi)=C+A\sum_{n=1}^{\infty}a(n)e^{2\pi inz/M},
    \end{align}
    where 
    \begin{align}
        A&=2(-2\pi i)^kG(\psi^0)/M^{k}(k-1)!,\\
        C&=\begin{cases}2L_M(k,\psi) &\chi:{\rm the~principal~character},\\
            0 &{\rm otherwise},
        \end{cases}\\
        a(n)&=\sum_{0<c\mid n}\chi(n/c)c^{k-1}\sum_{0<d\mid(l,c)}d\mu(l/d)\psi^0(l/d)\overline{\psi^0}(c/d).
    \end{align}
    Here $l=M/m_{\psi}$, $\mu$ is the M\"obius function, $L_M(k,\psi)=\sum_{n=1}^{\infty}\psi(n)n^{-k}$ is the Dirichlet series, and $G(\psi^0)$ is the Gauss sum of $\psi^0$.
\end{lemma}
\begin{example}Let $D$ be a fundamental discriminant and $\mathbf{1}$ be the principal character. Then
    \begin{align}
        E_{k}(z;\mathbf{1},\chi_D)=2L_{D}(k)+\frac{2(-2\pi i)^kG(\chi_D)}{(k-1)!|D|^k}\sum_{n=1}^{\infty}\left(\sum_{d\mid n}\legendre{D}{d}d^{k-1}\right)q^{2\pi i n z/|D|}.
    \end{align}
    \end{example}
\begin{example}If $D=D_1D_2$ is a product of relatively prime fundamental discriminants then 
\begin{align}
    E_{k}(z;\chi_{D_2},\chi_{D_1}):=C+\frac{2(-2\pi i)^kG(\chi_{D_1})}{|D_1|^k(k-1)!}\sum_{n=1}^{\infty}\left(\sum_{\substack{d_1,d_2>0\\ d_1d_2=n}}\legendre{D_1}{d_1}\legendre{D_2}{d_2}d_1^{k-1}\right)e^{2\pi inz/|D_1|},
\end{align}
where $C$ is zero unless $D_2=1$.
\end{example}
We shall compare our Eisenstein series $G_{k,D}(z)$ \eqref{eq:G_k,D} and $G_{k,D_1,D_2}(z)$, defined below in   \eqref{eq:gkd1d2def} \cite[p.~193]{Kohnen-Zagier1981} with the ones above given in Miyake \cite{MiyakeMFbook}. Comparing the Fourier coefficients of $G_{k,D}(z)$ and $E_{k}(z;\mathbf{1},\chi_D)$ gives
 \begin{align}
     G_{k,D}(z)=\frac{(k-1)!|D|^k}{2(-2\pi i)^kG(\chi_D)}E_{k}(|D|z,\mathbf{1},\chi_{D}).\label{eq:GkDandMiyakenotation}
 \end{align}
Recall that \cite[p.~193]{Kohnen-Zagier1981} for $D_1,D_2$ relatively prime fundamental discriminants with $(-1)^kD_1D_2>0$:
\begin{align}
    \label{eq:gkd1d2def}
    G_{k,D_1,D_2}(z)&=\sum_{n\geq0}\sigma_{k-1,D_1,D_2}(n)q^n,\\
    \sigma_{k-1,D_1,D_2}(n)&=\begin{cases}
        -L_{D_1}(1-k)L_{D_2}(0) &n=0,\\\sum\limits_{\substack{d_1,d_2>0\\ d_1d_2=n}}\legendre{D_1}{d_1}\legendre{D_2}{d_2}d_1^{k-1} &n>0.
    \end{cases}
\end{align}
where the constant term is zero unless $D_2=1$.  Hence by comparing the Fourier coefficients of $G_{k,D_1,D_2}(z)$ and $E_{k}(z;\chi_{D_2},\chi_{D_1})$, we get 
\begin{align}
    G_{k,D_1,D_2}(z)=\frac{|D_1|^k(k-1)!}{2(-2\pi i)^k G(\chi_{D_1})}E_{k}(|D_1|z;\chi_{D_2},\chi_{D_1}).\label{eq:GkD1D2prelim}
\end{align}
The following expression of $G_{k,D_1,D_2}(z)$ is useful for Lemma \ref{lem:FourierexpansionofGkD}.
\begin{lemma} Let $k\ge3$ and $D=D_1D_2$ be a product of  coprime fundamental discriminants. Then
\begin{align}\label{eq:GkD1D2}
    G_{k,D_1,D_2}(z)=\frac{|D_1|^k(k-1)!}{2(-2\pi i)^k G(\chi_{D_1})}\chi_{D_2}(|D_1|)\sideset{}{'}\sum_{\substack{m,n\in\mathbb{Z}\\D_1\mid m}}\frac{\chi_{D_2}(m)\chi_{D_1}(n)}{(mz+n)^k}.
\end{align}
\end{lemma}
\begin{proof}
  Note that 
  \begin{align}
      E_{k}(|D_1|z;\chi_{D_2},\chi_{D_1})&=\sideset{}{'}\sum\limits_{m,n\in\mathbb{Z}}\chi_{D_2}(m)\chi_{D_1}(n)(m|D_1|z+n)^{-k}\\&=\chi_{D_2}(|D_1|)\sideset{}{'}\sum_{\substack{\ell,n\in\mathbb{Z}\\D_1\mid\ell}}\chi_{D_2}(\ell)\chi_{D_1}(n)(\ell z+n)^{-k}.
  \end{align}
Thus the result follows from \eqref{eq:GkD1D2prelim}.
\end{proof}
Let $k\ge3$ and $\chi$ be a Dirichlet character mod $N$. We define the Eisenstein series for the cusp at infinity \cite[p.~272]{MiyakeMFbook} as 
\begin{align}
    E_{k,N}^{\ast}(z;\chi)=\sum_{\begin{bsmallmatrix}a&b\\c&d\end{bsmallmatrix}\in\Gamma_{\infty}\backslash\Gamma_0(N)}\frac{\chi(d)}{(cz+d)^k},
\end{align}
where $\Gamma_{\infty}=\{\pm\begin{bsmallmatrix}
    1 & n \\ 0 &1
\end{bsmallmatrix}:n\in\mathbb{Z}\}$. 

Now, we are ready to prove that $G_{k,4D}$ is an Eisenstein series for the cusp at infinity of level $4|D|$.
\begin{lemma}\label{lem:defEisensteininfity}Let $\mathbf{1}$ denote the principal Dirichlet character. Then 
\begin{align}
    2L_{N}(k,\chi)E_{k,N}^{\ast}(z;\chi)=E_{k}(Nz;\mathbf{1},\chi).
\end{align}
\end{lemma}
\begin{proof}
    Recall the following bijection \cite[Lemma 7.1.6]{MiyakeMFbook},
    \begin{equation*}
        \Gamma_\infty \backslash \Gamma_0(N) \rightarrow \{(c,d)\in \BZ^2:N\mid c,(c,d)=1,d>0\}:\Gamma_{\infty}\begin{bsmallmatrix}
            a &b\\c&d
        \end{bsmallmatrix}\mapsto (c,d),
    \end{equation*}
    which allows us to rewrite $E^*_{k,N}(z;\chi)$ as
    \begin{equation*}
        E^*_{k,N}(z;\chi)=\sum_{\substack{(c,d)=1\\d>0,N\mid c}} \frac{\chi(d)}{(cz+d)^k}=\frac{1}{2}\sum_{\substack{(c,d)=1\\ N\mid c}} \frac{\chi(d)}{(cz+d)^k}.
    \end{equation*}
    It follows that
    \begin{align}
        2L_N(k,\chi)E^*_{k,N}(z;\chi)&=2\sum_{n=1}^\infty \frac{\chi(n)}{n^k}\left(\frac{1}{2}\sum_{\substack{(c,d)=1\\ N\mid c}} \frac{\chi(d)}{(cz+d)^k}\right)\\
        &=\sum_{n=1}^\infty \sum_{\substack{(c,d)=1\\ N\mid c}} \frac{\chi(nd)}{(ncz+nd)^k} \\
        &=\sideset{}{'}\sum_{\substack{c,d\in\mathbb{Z}\\N\mid c}}\frac{\chi(d)}{(cz+d)^{k}} 
        \\&=\sideset{}{'}\sum_{c,d\in \ZZ} \frac{\chi(d)}{(Ncz+d)^k}
        \\&=E_k(Nz;\mathbf{1},\chi),
    \end{align}
    as desired. 
\end{proof}
From \eqref{eq:GkDandMiyakenotation} and Lemma \ref{lem:defEisensteininfity},  we know that $G_{k,D}(z)$ is an Eisenstein series at infinity. Comparing the constant terms of the Fourier expansions of $G_{k,D}(z)$ and $E_{k,|D|}^{\ast}(z;\chi_D)$ gives
\begin{align}
    G_{k,D}(z)=\frac{L_D(1-k)}{2}E_{k,|D|}^{\ast}(z;\chi_D).\label{eq:GkDandcuspinfty}
\end{align}
Note also that \eqref{eq:GkDandMiyakenotation} and the proof of Lemma 4.5 imply that
\begin{align}\label{eq:GkDforcomputingFourier}
    G_{k,D}(z)=\frac{(k-1)!|D|^k}{2(-2\pi i)^kG(\chi_D)}\sideset{}{'}\sum_{\substack{c,d\in\ZZ\\D\mid c}}\frac{\chi_D(d)}{(cz+d)^k}.
\end{align}
In fact, equation \eqref{eq:GkDforcomputingFourier} will be more convenient for us to compute the Fourier expansion of $G_{k,D}(z)$ at different cusps. We need the following lemma; see also \cite[p.~271]{GZ-86}.
\begin{lemma}\label{lem:Gk4DisEisensteinatinfty}
    Let $L_D^{(4)}(k)=\sum\limits_{\substack{(n,4)=1\\n\geq1}}\chi_D(n)n^{-k}$. Then
    \begin{align}
        E_{k,4|D|}^{\ast}(z;\chi_D)=\frac{L_D(k)}{L_{D}^{(4)}(k)}\left(E_{k,|D|}^{\ast}(4z;\chi_D)-2^{-k}\legendre{D}{2}E_{k,|D|}^{\ast}(2z;\chi_D)\right).
    \end{align}
\end{lemma}
\begin{proof}
    Observe that
    \begin{align}
        2L_D^{(4)}(k)E_{k,4|D|}^{\ast}(z; \chi_D)
        &=2\sum_{\substack{n\ge 1\\(4,n)=1}}\frac{\chi_D(n)}{n^k}\left(\frac{1}{2}\sum_{\substack{(c,d)=1\\4D\mid c}}\frac{\chi_D(d)}{(cz+d)^{k}}\right)\\
        &=\sum_{\substack{n\ge 1\\(4,n)=1}} \sum_{\substack{(c,d)=1\\4D\mid c}}\frac{\chi_D(nd)}{(ncz+nd)^k}\\ &=\sum_{\substack{(d',4D)=1\\4D\mid c'}}\frac{\chi_D(d')}{(c'z+d')^k},
\end{align}
where $nc=c'$ and $nd=d'$. Note that we can replace $(d',4D)=1$ by $(d',4)$ since $\chi_D(d')=0$ otherwise. It follows that
\begin{align}
      2L_D^{(4)}(k)E_{k,4|D|}^{\ast}(z; \chi_D)
       &=\sum_{\substack{(d, 4)=1\\4D\mid c}} \frac{\chi_D(d)}{(cz+d)^k}\\
        &=\sideset{}{'}\sum_{\substack{c,d\in\mathbb{Z}\\4D|c}} \left(\sum_{e|(d,4),e>0}\mu(e)\right)\frac{\chi_D(d)}{(cz+d)^k}\\
        &=\sum_{e\mid4,e>0}\mu(e)\sideset{}{'}\sum_{\substack{c,d\in \mathbb{Z}\\4D\mid c, e\mid d}}\frac{\chi_D(d)}{(cz+d)^k},
       \end{align} 
where we used $\sum\limits_{e\mid(d,4)}\mu(e)=0$ for $(d,4)>1$ in the second equality. Substituting $d=ey$ and $c=4x$,
\newline
\begin{align} 2L_D^{(4)}(k)E_{k,4|D|}^{\ast}(z; \chi_D)&=\sum_{e\mid 4,e>0}\mu(e) \sideset{}{'}\sum_{\substack{x,y\in\mathbb{Z}\\D\mid x}}\frac{\chi_D(ey)}{(4xz+ey)^k}\\&=\sum_{e\mid 4,e>0}\mu(e)e^{-k}\chi_D(e) \sideset{}{'}\sum_{\substack{x,y\in\mathbb{Z}\\D|x}}\frac{\chi_D(y)}{(x4z/e+y)^k}
\\ &=\sum_{e\mid 4,e>0}\mu(e)e^{-k}\chi_D(e)2L_D(k)E_{k,D}^{\ast}(4z/e,\chi_D)\\&=2L_D(k)\left(E_{k,|D|}^{\ast}(4z;\chi_D)-2^{-k}\legendre{D}{2}E_{k,|D|}^{\ast}(2z;\chi_D)\right),
    \end{align}
    where the second to last equality is from the proof of Lemma \ref{lem:defEisensteininfity}.
\end{proof}
From Lemma \ref{lem:Gk4DisEisensteinatinfty} and \eqref{eq:GkDandcuspinfty}, we know that $G_{k,4D}(z)$ is an Eisenstein series for the cusp at infinity in $M_{k}(4|D|,\chi_D)$. Comparing the constants of $G_{k,4D}(z)$ and $E_{k,4D}^{\ast}(z;\chi_D)$, we have
\begin{align}
    G_{k,4D}(z)=\frac{L_{D}(1-k)}{2}\left(1-2^{-k}\left(\frac{D}{2}\right)\right)E_{k,4D}^{\ast}(z;\chi_D).\label{eq:Gk4Dandcuspinfty}
\end{align}

\section{The Rankin-Selberg convolution}\label{sec:rankin}
The purpose of this section is to prove Propositions \ref{prop:FdkeInnerProduct} and \ref{prop:RankinSelberghalfintegral}. For two elements $f$ and $g$ of $M_k(N)$ such that $fg$ is a cuspform, the Petersson inner product is given by 
\begin{align}
    \langle f,g\rangle_{\Gamma_0(N)}=\int_{\Gamma_{0}(N)\backslash\mathbb{H}}f(z)\overline{g(z)}\im(z)^kd\mu,\label{eq:defofPeterssoninnerproduct}
\end{align}
where $z=x+iy$ and $d\mu=dxdy/y^2$. We use $\langle\cdot,\cdot\rangle$ to denote $\langle\cdot,\cdot\rangle_{\Gamma_0(N)}$ if the level is clear from the context. For $f(z)=\sum_{n\geq1}a_f(n)q^n\in S_{k}(N,\chi)$, we put $f_{\rho}(z):=\sum_{n\geq 1}\overline{a_f(n)}q^n$. Note that $f_{\rho}(z)=f(z)$ if $f$ is a newform.

We now review the classical result on the Rankin-Selberg convolution, which was reformulated and generalized in Zagier \cite{Zagier1976}, keeping in mind the difference between our definition of the Rankin-Cohen bracket and the one used therein.
\begin{lemma}[{\cite[Propsition 6]{Zagier1976}\label{lem:RankinSelbergZagier}}]
Let $k_1$ and $k_2$ be real numbers with $k_2\geq k_1+2>2$. Let $f(z)=\sum_{n=1}^{\infty}a(n)q^n$ and $g(z)=\sum_{n=0}^{\infty}b(n)q^n$ be modular forms in $S_k(N,\chi)$ and $M_{k_1}(N,\chi_1)$, where $k=k_1+k_2+2e, e\geq0$ and $\chi=\chi_1\chi_2$. 
Then
    \begin{align}
        \langle f,[g,E^*_{k_2,N}(\cdot;\chi_2)]_e\rangle=\frac{(-1)^e}{e!}\frac{\Gamma(k-1)\Gamma(k_2+e)}{(4\pi)^{k-1}\Gamma(k_2)}\sum_{n=1}^{\infty}\frac{a(n)\overline{b(n)}}{n^{k_1+k_2+e-1}}.\label{eq:RankinSelberg}
    \end{align}
\end{lemma}
To obtain Proposition \ref{prop:FdkeInnerProduct}, we need to deal with the case $k_1=k_2$, which can be done by following Shimura \cite{Shimura1976} and Lanphier's work \cite{ShimuraoperatorLanphier}. For $f(z)=\sum_{n=1}^{\infty}a(n)q^n\in S_{k}(N,\chi)$ and $g(z)=\sum_{n=0}^{\infty}b(n)q^n\in M_{\ell}(N,\psi)$, we put
\[D(s,f,g)=\sum_{n=1}^{\infty}a(n)b(n)n^{-s},\quad \re(s)\gg0.\]
We are particularly interested in the following case. 
\begin{lemma}\label{lem:defofD}
    Let $f=\sum_{n=1}^\infty a(n)q^n\in S_{2\ell}(1)$ be a normalized eigenform with $\ell=k+2e, e>0$ and $k\geq4$ integers, and let $D$ be an odd fundamental discriminant. Then
    \begin{align}
        D(s,f,G_{k,D})=\frac{L(f,s)L(f,D,s-k+1))}{L_{D}(2s-3k-4e+2)},\quad \re(s)\gg0.
    \end{align}
\end{lemma}
\begin{proof}
    Note that for $\re(s)\gg0$, we have
\begin{align}D(s,f,G_{k,D})=\sum_{n=1}^{\infty}\frac{\sigma_{k-1,\mathbf{1},\chi_D}(n)a(n)}{n^s},
     \end{align}
     where $\sigma_{k-1,\mathbf{1},\chi_D}(n)=\sum_{d\mid n}\chi_{D}(d)d^{k-1}$. A standard computation (see \cite[Proposition 4.1]{Raumproducts}) gives 
     \begin{align}
         \sum_{n=1}^{\infty}\frac{\sigma_{k-1,\mathbf{1},\chi_D}(n)a(n)}{n^s}=\frac{L(f,s)L(f,D,s-(k-1))}{L_{D}(2s-(k-1)+1-(2k+4e))},
     \end{align}
     as desired.
\end{proof}
From Shimura \cite[p.~786-789]{Shimura1976}, $D(s,f,G_{k,D})$  has a meromorphic continuation to the whole complex plane and  $D(s,f,G_{k,D})$ is holomorphic at $s=2k+2e-1$, see \cite[p.~789]{Shimura1976}.

The Maass-Shimura operators \cite[p.~788 (2.8)]{Shimura1976} are defined by
\begin{align}
    \delta_{\lambda}&=\frac{1}{2\pi i}\left(\frac{\lambda}{2i y}+\frac{\partial}{\partial z}\right),\quad0<\lambda\in\mathbb{Z},\\
    \delta_{\lambda}^{(r)}&=\delta_{\lambda+2r-2}\cdots\delta_{\lambda+2}\delta_{\lambda},\quad 0\leq r\in\mathbb{Z},
\end{align}
where we understand that $\delta^{(0)}_{\lambda}$ is the identity operator. 
A relation between Maass-Shimura operators and the Rankin-Cohen bracket is given by
\begin{align}
    \delta_k^{(n)}f(z)\times g(z)=\sum_{j=0}^n\frac{(-1)^j\binom{n}{j}\binom{k+n-1}{n-j}}{\binom{k+\ell+2j-2}{j}\binom{k+\ell+n+j-1}{n-j}}\delta^{(n-j)}_{k+\ell+2j}[f,g]_j(z),\label{eq:ShimuraoperatorandRankin-Cohenbracket}
\end{align}
where $f\in M_{k}(\Gamma)$ and $g\in M_{\ell}(\Gamma)$ for any congruence subgroup $\Gamma$; see \cite[Theorem 1]{ShimuraoperatorLanphier}.

We recall the following two results.
\begin{lemma}[{\cite[Lemma 6]{Shimura1976}}]\label{lem:shimuralemma6}
   Suppose $f\in S_{k}(N,\chi), g\in M_{l}(N,\overline{\chi})$ and $k=l+2r$ with a positive integer $r.$ Then $\langle \delta^{(r)}g,f_{\rho}\rangle=0$. 
\end{lemma}
\begin{lemma}[{\cite[Theorem 2]{Shimura1976}}] \label{lem:shimuraD}
Suppose $f\in S_{2\ell}(|D|)$ with $\ell=k+2e, e>0$ and $k\geq4$, and $D$ is a fundamental discriminant. Then
    \begin{align}D(2k+4e-1-2e),f,G_{k,D})=c\pi^{2k+4e-1}\langle G_{k,D}\delta_{k}^{(2e)}E_{k,|D|}^{\ast}(z;\chi_D),f_{\rho} \rangle,\end{align}
    where $\langle\cdot,\cdot\rangle$ denotes the non-normalized Petersson inner product on defined in \eqref{eq:defofPeterssoninnerproduct} and 
    \begin{align}
        c&=\frac{\Gamma(2k+4e-k-2(2e))}{\Gamma(2k+4e-1-2e)\Gamma(2k+4e-k-2e)}(-1)^{2e}4^{2k+4e-1}. \\
        &=\frac{\Gamma(k)}{\Gamma(2k-1+2e)\Gamma(k+2e)}4^{2k+4e-1}
    \end{align}
\end{lemma}

We apply these two results in our situation to obtain the following.
\begin{proposition}\label{prop:RankinSelbergatcentralvalue}
    Let $f\in S_{2\ell}(1)$ be a normalized eigenform with $\ell=k+2e, e>0$ and $k\geq4$. Then
    \begin{align}
        \langle[G_{k,D},G_{k,D}]_{2e},f\rangle_{\Gamma_0(|D|)}=\frac{1}{2}\frac{\Gamma(2k+4e-1)\Gamma(k+2e)}{(2e)!(4\pi)^{2k+4e-1}\Gamma(k)}\frac{L_{D}(1-k)}{L_D(k)}L(f,2k+2e-1)L(f,D,k+2e).
    \end{align}
\end{proposition}
\begin{proof}
Note that $f_{\rho}=f$ since $f$ is a normalized eigenform. Lemma \ref{lem:shimuraD} gives
\begin{align}
 \langle G_{k,D}\delta_{k}^{(2e)}E_{k,|D|}^{\ast}(z;\chi_D),f\rangle_{\Gamma_0(|D|)}= \frac{\Gamma(2k+2e-1)\Gamma(k+2e)}{(4\pi)^{2k+4e-1}\Gamma(k)}D(2k+2e-1,f,G_{k,D}).
\end{align}
By Lemma \ref{lem:shimuralemma6} and \eqref{eq:ShimuraoperatorandRankin-Cohenbracket}, 
\begin{align}\langle G_{k,D}\delta_{k}^{(2e)}E_{k,|D|}^{\ast}(z;\chi_D),f\rangle_{\Gamma_0(|D|)}=\frac{1}{\binom{2k+4e-2}{2e}}\langle [E_{k,|D|}^{\ast}(z;\chi_D),G_{k,D}]_{2e},f\rangle,
\end{align}
which implies that
\begin{align}
   \langle [E_{k,|D|}^{\ast}(z;\chi_D),G_{k,D}]_{2e},f\rangle_{\Gamma_0(|D|)}&= \frac{\binom{2k+4e-2}{2e}\Gamma(2k+2e-1)\Gamma(k+2e)}{(4\pi)^{2k+4e-1}\Gamma(k)}D(2k+2e-1,f,G_{k,D}).
\end{align}
Since $G_{k,D}(z)=\frac{L_D(1-k)}{2}E_{k,|D|}^{\ast}(z;\chi_D)$ \eqref{eq:GkDandcuspinfty} and by Lemma \ref{lem:defofD}, we have
 \begin{align}
        \langle [G_{k,D},G_{k,D}]_{2e},f\rangle_{\Gamma_0(|D|)}&=\frac{L_D(1-k)}{2}\frac{\binom{2k+4e-2}{2e}\Gamma(2k+2e-1)\Gamma(k+2e)}{(4\pi)^{2k+4e-1}\Gamma(k)}D(2k+2e-1,f,G_{k,D})\\&=\frac{1}{2}\frac{\Gamma(2k+4e-1)\Gamma(k+2e)}{(2e)!(4\pi)^{2k+4e-1}\Gamma(k)}\frac{L_{D}(1-k)}{L_D(k)}L(f,2k+2e-1)L(f,D,k+2e),
    \end{align}
as desired.
\end{proof}
Now we  prove Propositions \ref{prop:FdkeInnerProduct} and \ref{prop:RankinSelberghalfintegral}, which generalize \cite[Proposition 1]{Kohnen-Zagier1981} and \cite[Proposition 2]{Kohnen-Zagier1981}, respectively.
\begin{proposition}\label{prop:FdkeInnerProduct}
    Let $f=\sum_{n=1}^{\infty} a(n)q^n$ be a normalized eigenform in $S_{2\ell}(1)$ with $\ell=k+2e, e>0$ and $k\geq4$, and let $D$ be an odd fundamental discriminant with $(-1)^\ell D>0$.  Then
    \begin{align}
       \langle \mathcal{F}_{D,k,e},f\rangle_=\frac{1}{2}\frac{\Gamma(2k+4e-1)\Gamma(k+2e)}{(2e)!(4\pi)^{2k+4e-1}\Gamma(k)}\frac{L_{D}(1-k)}{L_D(k)}L(f,2k+2e-1)L(f,D,k+2e).
    \end{align}
\end{proposition}
\begin{proof}
Recall that \eqref{eq:F_D,k,e}
\begin{align}
\mathcal{F}_{D,k,e}(z)=\Tr^D_1[G_{k,D}(z),G_{k,D}(z)]_e.
\end{align}
As $\langle f,g\rangle_{\Gamma_0(M)}=\langle f,\Tr_N^M g\rangle_{\Gamma_0(N)}$ for $N\mid M$, for $f\in S_k(N),g\in M_{k}(M)$ (see \cite[p.~271]{GZ-86}), we get
    \begin{align}
\langle\mathcal{F}_{D,k,e},f\rangle&=
 \langle[G_{k,D}(z),G_{k,D}(z)]_{2e},f\rangle_{\Gamma_0(|D|)}.\end{align}
  Then the result follows from Proposition \ref{prop:RankinSelbergatcentralvalue}. 
\end{proof}

\begin{proposition}\label{prop:RankinSelberghalfintegral}
    Let $g=\sum c_g(n)q^n\in S^{+}_{\ell+1/2}(4)$ be a Hecke eigenform and  $f\in S_{2\ell}(1)$ be the normalized Hecke eigenform corresponding to it by the Shimura correspondence, where $\ell=k+2e, e>0$ and $k\geq4$. Let $D$ be an odd fundamental discriminant with $(-1)^\ell D>0$. Then
    \begin{align}
        \langle g,\mathcal{G}_{D,k,e}\rangle=\frac{3}{2}\frac{\Gamma(k+2e-\frac{1}{2})\Gamma(k+e)}{e!(4\pi)^{k+2e-1/2}\Gamma(k)}\frac{L_D(1-k)}{L_D(k)}|D|^{-k-e+1/2}L(f,2k+2e-1)c_g(|D|),
    \end{align}
    where the Petersson inner product is  $\langle g,\mathcal{G}_{D,k,e}\rangle:=\int_{\Gamma_0(4)\backslash\mathbb{H}}g(z)\overline{\mathcal{G}_{D,k,e}(z)}\im(z)^{k+2e+1/2}d\mu$. 
\end{proposition}
\begin{proof}
    Recall that  $\mathcal{G}_{D, k, e}$ is given in \eqref{eq:G_D,k,e}:
    \begin{equation}
\mathcal{G}_{D, k, e}(z)=\frac{3}{2}\left(1-2^{-k}\left(\frac{D}{2}\right)\right)^{-1}\pr^{+}\Tr^{4D}_{4}[G_{k,4D}(z),\theta(|D|z)]_e. 
\end{equation}
    Since $\pr^{+}$ \eqref{eq:defofprojection} is the projection from $M_{\ell+1/2}(4)$ to $M^{+}_{\ell+1/2}(4)$,

    we have
    \begin{align}
    \langle g,\mathcal{G}_{D,k,e}\rangle&=\frac{3}{2}\left(1-2^{-k}\left(\frac{D}{2}\right)\right)^{-1}\langle \pr^{+}g,\Tr^{4D}_{4}([G_{k,4D}(z),\theta(|D|z)]_e\rangle\\&= \frac{3}{2}\left(1-2^{-k}\left(\frac{D}{2}\right)\right)^{-1}\langle g,\Tr^{4D}_{4}([G_{k,4D}(z),\theta(|D|z)]_e\rangle\\&
    = \frac{3}{2}\left(1-2^{-k}\left(\frac{D}{2}\right)\right)^{-1}\langle g,([G_{k,4D}(z),\theta(|D|z)]_e\rangle_{\Gamma_0(4|D|)}
    \\&=\frac{3}{4}L_{D}(1-k)\langle g(z),[E_{k,4D}^{\ast}(z;\chi_D),\theta(|D|z)]_e\rangle_{\Gamma_0(4|D|)}\\&=\frac{3(-1)^e}{4}L_{D}(1-k)\langle g(z),[\theta(|D|z),E_{k,4D}^{\ast}(z;\chi_D)]_e\rangle_{\Gamma_0(4|D|)},\end{align}
    where we used \eqref{eq:Gk4Dandcuspinfty} in the second to last equality. Now Lemma \ref{lem:RankinSelbergZagier} gives
    \begin{align}
        \langle g,\mathcal{G}_{D,k,e}\rangle&=\frac{3(-1)^e}{4}L_{D}(1-k)\frac{(-1)^e}{e!}\frac{\Gamma(k+2e-\frac{1}{2})\Gamma(k+e)}{(4\pi)^{k+2e-1/2}\Gamma(k)}\sum_{n=1}^{\infty}\frac{2c_g(n^2|D|)}{(|D|n^2)^{k+e+1/2-1}}\\&=\frac{3}{2}\frac{\Gamma(k+2e-\frac{1}{2})\Gamma(k+e)}{e!(4\pi)^{k+2e-1/2}\Gamma(k)}L_D(1-k)|D|^{-(k+e-1/2)}\sum_{n=1}^{\infty}\frac{c_g(n^2|D|)}{n^{2k+2e-1}}.\label{eq:middlestepofinnerprodGkDe}
    \end{align}
By \cite[Theorem 1 (ii)]{Kohenhalfintegralweight1980}, we get
\begin{align}
    L_{D}(s-(k+2e)+1)\sum_{n=1}^{\infty}\frac{c_g(n^2|D|)}{n^{2k+2e-1}}=c_g(|D|)L(f,s),
\end{align}
which implies that
\begin{align}
   \langle g,\mathcal{G}_{D,k,e}\rangle=\frac{3}{2}\frac{\Gamma(k+2e-\frac{1}{2})\Gamma(k+e)}{e!(4\pi)^{k+2e-1/2}\Gamma(k)}\frac{L_D(1-k)}{L_D(k)}|D|^{-k-e+1/2}L(f,2k+2e-1)c_g(|D|),
\end{align}
as desired. 
\end{proof}

\section{Fourier expansions}\label{sec:fourier}
In this section, we compute the Fourier coefficients needed for the proof of Theorem \ref{thm:lift}. It is convenient to have explicit formulas for  $G_{k,D}$ and $\theta$ under the action of certain matrices in $\Sl_2(\mathbb{Z})$, which we do in Lemmas \ref{lem:FourierexpansionofGkD} and \ref{lem:Fourierexpansiontheta}. Propositions \ref{prop:FourierexFDE} and \ref{prop:FourierexofGDe} then give formulas for $\mathcal F_{D,k,e}$ and $\mathcal G_{D,k,e}$, which we use in the final computation of the Fourier coefficients carried out in Lemmas \ref{lem:ffouriercoefficient} and \ref{eq:gfouriercoefficients}.
\begin{lemma}\label{lem:FourierexpansionofGkD} Let $k\ge3$. Suppose $D$ is an odd fundamental discriminant and $D=D_1D_2$ is a product of two fundamental discriminants. Let $\gamma=\begin{bsmallmatrix}
    a&b\\c&d
\end{bsmallmatrix}\in\Sl_2(\mathbb{Z})$ with $\gcd(c,D)=|D_1|$. Then
\begin{align}
    G_{k,D}(z)\bigg|_k\begin{bmatrix}
        a & b\\ c& d\end{bmatrix}=\legendre{D_2}{c}\legendre{D_1}{d|D_2|}\legendre{D_2}{|D_1|}\frac{\varepsilon_{|D_1|}}{\varepsilon_{|D|}}|D_2|^{-1/2}G_{k,D_1,D_2}\left(\frac{z+c^{\ast}d}{|D_2|}\right),
\end{align}
where $c^{\ast}$ is an integer with $cc^{\ast}\equiv 1\pmod{|D_2|}$, and  $\varepsilon_n$ is given by 
\begin{equation}\label{eq:epsilonn}
    \varepsilon_n=\begin{cases}
        1 & n\equiv 1\pmod4, \\
        i & n\equiv 3\pmod4.
    \end{cases}
\end{equation}

\end{lemma}
\begin{proof}
 We follow the idea in Gross-Zagier \cite[pp.~273-275]{GZ-86}.   By equation \eqref{eq:GkDforcomputingFourier}, we have
    \begin{align}
    \frac{2(-2\pi i)^kG(\chi_D)}{(k-1)!|D|^k}G_{k,D}(z)\bigg|_k\begin{bmatrix}
        a&b\\ c&d\end{bmatrix}&=\sideset{}{'}\sum_{\substack{l,r\in\mathbb{Z}\\D\mid l}}\frac{\chi_D(r)}{(l(az+b)+r(cz+d))^k}\\&=\sideset{}{'}\sum_{\substack{m,n\in\ZZ \\md\equiv nc~\text{mod}~|D|}}\frac{\chi_D(an-bm)}{(mz+n)^k},
    \end{align}
    where $(m,n)=(l,r)\begin{bsmallmatrix}
        a&b\\ c&d\end{bsmallmatrix}$. Since $md\equiv nc\pmod{|D|}$, we have 
        \begin{align}
            d(an-bm)&=adn-bmd\equiv adn-bcn\equiv n\pmod{|D|},\label{eq:dealwithchiD1}\\ c(an-bm)&=anc-bcm\equiv adm-bcm\equiv m\pmod{|D|}.\label{eq:dealwithchiD2}
        \end{align}
Note also that $\gcd(D_1,D_2)=1$. Then \eqref{eq:dealwithchiD1} and \eqref{eq:dealwithchiD2} imply that
 \begin{align}
     \chi_D(an-bm)&=\chi_{D_1}(an-bm)\chi_{D_2}(an-bm)\\&=\chi_{D_1}(d)\chi_{D_1}(n)\chi_{D_2}(c)\chi_{D_2}(m).
 \end{align}
 Since $D_1,D_2\mid(md-nc)$, $(d,D_1)=1, (c,D_2)=1$ and $(c,D)=|D_1|$, we must have $D_1\mid m$; and $n\equiv c^{\ast}md\pmod{|D_2|}$. By the Chinese Remainder Theorem, we can choose $c^{\ast}$ such that $D_1\mid c^{\ast}$. Now we write $n=c^{\ast}md+l|D_2|$. It follows that
 \begin{align}
 \frac{2(-2\pi i)^kG(\chi_D)}{(k-1)!|D|^k}G_{k,D}(z)\bigg|_k\begin{bmatrix}
        a&b\\ c&d\end{bmatrix}&=\sideset{}{'}\sum_{\substack{m,l\in\mathbb{Z}\\ D_1\mid m}}\frac{\chi_{D_1}(d)\chi_{D_1}(c^{\ast}md+l|D_2|)\chi_{D_2}(c)\chi_{D_2}(m)}{(mz+mc^{\ast}d+l|D_2|)^k}\\&=\chi_{D_2}(c)\chi_{D_1}(d|D_2|)\sideset{}{'}\sum_{\substack{m,l\in\mathbb{Z}\\ D_1\mid m}}\frac{\chi_{D_2}(m)\chi_{D_1}(l)}{(mz+mc^{\ast}d+l|D_2|)^k}\\&=\chi_{D_2}(c)\chi_{D_1}(d|D_2|)|D_2|^{-k}\sideset{}{'}\sum_{\substack{m,l\in\mathbb{Z}\\ D_1\mid m}}\frac{\chi_{D_2}(m)\chi_{D_1}(l)}{\left(m\frac{z+c^{\ast}d}{|D_2|}+l\right)^k}.\label{eq:middle1}
 \end{align}
Note that \eqref{eq:GkD1D2} implies that 
\begin{align}
    \sideset{}{'}\sum_{\substack{m,l\in\mathbb{Z}\\ D_1\mid m}}\frac{\chi_{D_2}(m)\chi_{D_1}(l)}{\left(m\frac{z+c^{\ast}d}{|D_2|}+l\right)^k}=\frac{2(-2\pi i)^kG(\chi_{D_1})}{|D_1|^k(k-1)!}\chi_{D_2}(|D_1|)G_{k,D_1,D_2}\left(\frac{z+c^{\ast}d}{|D_2|}\right).\label{eq:middle2}
\end{align} 
Plugging \eqref{eq:middle2} into \eqref{eq:middle1} gives
 \begin{align}
          G_{k,D}(z)\bigg|_k\begin{bmatrix}
        a & b\\ c& d\end{bmatrix}&=\chi_{D_2}(c)\chi_{D_1}(d|D_2|)\chi_{D_2}(|D_1|)\frac{G(\chi_{D_1})}{G(\chi_D)}G_{k,D_1,D_2}\left(\frac{z+c^{\ast}d}{|D_2|}\right). \label{eq:fortheuseofFDE}
 \end{align}
From \cite[Proposition 2.2.24. p.~49]{Cohenbook} we know that 
\begin{align}
  G(\chi_{D_1})=\varepsilon_{|D_1|}|D_1|^{1/2}\quad{\rm and}\quad G(\chi_{D})=\varepsilon_{|D|}|D|^{1/2},
\end{align}
which implies that 
\begin{align}
    G_{k,D}(z)\bigg|_k\begin{bmatrix}
        a & b\\ c& d\end{bmatrix}=\legendre{D_2}{c}\legendre{D_1}{d|D_2|}\legendre{D_2}{|D_1|}\frac{\varepsilon_{|D_1|}}{\varepsilon_{|D|}}|D_2|^{-1/2}G_{k,D_1,D_2}\left(\frac{z+c^{\ast}d}{|D_2|}\right),
\end{align}
as desired.
\end{proof}
\begin{lemma}\label{lem:Fourierexpansiontheta}Let $D$ be an odd fundamental discriminant and $D=D_1D_2$ be a product of two fundamental discriminants. Then
\begin{align}
    \theta(z)\bigg|_{\frac{1}{2}}\begin{bmatrix}
        |D|&0\\4|D_1|&1
\end{bmatrix}=\varepsilon_{|D_2|}^{-1}|D|^{1/4}|D_2|^{-1/2}\theta\left(\frac{|D_1|z+4^{\ast}}{|D_2|}\right),
\end{align}
where $4^{\ast}$ is an integer such that $44^{\ast}\equiv1\pmod{|D_2|}$.
\end{lemma}
\begin{proof}
    Since $(4,D_2)=1$, there exist $n,m\in\mathbb{Z}$ such that $n|D_2|+4m=1$ and 
    \begin{align}
        \begin{bmatrix}
            |D|&0\\4|D_1|&1
        \end{bmatrix}=\begin{bmatrix}
            |D_2|&-m\\4&n
        \end{bmatrix}\begin{bmatrix}
            |D_1|&m\\0&|D_2|
        \end{bmatrix}.
    \end{align}
    It follows that 
    \begin{align}
        \theta(z)\bigg|_{\frac{1}{2}}\begin{bmatrix}
        |D|&0\\4|D_1|&1
    \end{bmatrix}=\theta(z)\bigg|_{\frac{1}{2}}\begin{bmatrix}
            |D_2|&-m\\4&n
        \end{bmatrix}\begin{bmatrix}
            |D_1|&m\\0&|D_2|
        \end{bmatrix}.
    \end{align}
    Recall that the transformation law for $\theta$ (see e.g \cite[p.~148]{Koblitzbook}) gives
    \begin{align}
        \theta(z)\bigg|_{\frac{1}{2}}\begin{bmatrix}
            |D_2|&-m\\4&n
        \end{bmatrix}=\left(\frac{4}{n}\right)\varepsilon_{n}^{-1}\theta(z),
    \end{align}
    where $\varepsilon_n$ is as in \eqref{eq:epsilonn}. Since $n|D_2|+4m=1$ and $D_2\equiv1\pmod4$, we have $\varepsilon_{n}=\varepsilon_{|D_2|}$. Hence
        \begin{align}\theta(z)\bigg|_{\frac{1}{2}}\begin{bmatrix}
        |D|&0\\4|D_1|&1
    \end{bmatrix}=&
        \varepsilon_{|D_2|}\theta(z)\bigg|_{1/2}\begin{bmatrix}
            |D_1|&m\\0&|D_2|
\end{bmatrix}\\=&\varepsilon_{|D_2|}|D|^{1/4}|D_2|^{-1/2}\theta\left(\frac{|D_1|z+4^{\ast}}{|D_2|}\right),
    \end{align}
    which gives the desired result. 
\end{proof}

Next, we give some computations generalizing the lemma in \cite[p.~193]{Kohnen-Zagier1981}.

\begin{proposition}\label{prop:FourierexFDE}
Let $k\ge4$ and $e>0$ with $\ell=k+2e$ and let $D$ be an odd fundamental discriminant with $(-1)^{\ell}D>0$. Then
    \begin{align}
        \mathcal{F}_{D,k,e}(z)=\sum_{D=D_1D_2}\legendre{D_2}{-1}|D_{2}|^{-2e}U_{|D_2|}([G_{k,D_1,D_2}(z),G_{k,D_1,D_2}(z)]_{2e}),
    \end{align}
    where the summation is over all decompositions of $D$ as a product of two fundamental discriminants, and $U_{|D_2|}$ is the operator defined in \eqref{eq:defofUmap}.
\end{proposition}

\begin{proof}
    We consider the following system of representatives (Lemma \ref{lem:cosetreps}) of $\Gamma_0(|D|)\backslash \SL{2}(\ZZ)$, 
    \begin{equation*}
        \left\{\begin{bmatrix}
            1 & 0 \\ |D_1| & 1
        \end{bmatrix}\begin{bmatrix}
            1 & \mu \\ 0 & 1
        \end{bmatrix} \qquad \text{where } D=D_1D_2, \quad \mu \text{ mod }|D_2|\right\}
    \end{equation*}
    and $D_1,D_2$ are fundamental discriminants. By \eqref{eq:fortheuseofFDE} we have
    \begin{equation*}
       G_{k,D}(z)\bigg|_{k}\begin{bmatrix}
            1 & 0 \\ |D_1| & 1
        \end{bmatrix}\begin{bmatrix}
            1 & \mu \\ 0 & 1
        \end{bmatrix} =\legendre{D_2}{|D_1|}\legendre{D_1}{|D_2|}\legendre{D_2}{|D_1|}\frac{G(\chi_{D_1})}{G(\chi_D)}G_{k,D_1,D_2}\left(\frac{z+\mu+|D_1|^*}{|D_2|}\right)
    \end{equation*}
    where $|D_1|^*|D_1|=1 \mod |D_2|$. We then compute $\mathcal{F}_{D,k,e}(z)$, which is 
    \begin{align*}
        &\Tr_1^D([G_{k,D}(z),G_{k,D}(z)]_{2e}) \\
        =&\sum_{D_1D_2=D}\sum_{\mu \text{ mod } |D_2|} \left[G_{k,D}(z),G_{k,D}(z)\right]_{2e}\bigg|_{2k+4e}\begin{bmatrix}
            1 & 0 \\ |D_1| & 1
        \end{bmatrix}\begin{bmatrix}
            1 & \mu \\ 0 & 1
        \end{bmatrix} \\
        =&\sum_{D_1D_2=D}\sum_{\mu \text{ mod } |D_2|}\left[G_{k,D}(z)\bigg|_{k}\begin{bmatrix}
            1 & 0 \\ |D_1| & 1
        \end{bmatrix}\begin{bmatrix}
            1 & \mu \\ 0 & 1
        \end{bmatrix},G_{k,D}(z)\bigg|_k\begin{bmatrix}
            1 & 0 \\ |D_1| & 1
\end{bmatrix}\begin{bmatrix}
            1 & \mu \\ 0 & 1
        \end{bmatrix}\right]_{2e} \\
        =&\sum_{D_1D_2=D}\sum_{\mu \text{ mod } |D_2|}\legendre{D_2}{-1}|D_2|^{-1}\left[G_{k,D_1,D_2}\left(\frac{z+\mu+|D_1|^*}{|D_2|}\right),G_{k,D_1,D_2}\left(\frac{z+\mu+|D_1|^*}{|D_2|}\right)\right]_{2e},
    \end{align*}
where we used the well-known fact
$G(\chi_{D_1})^2=\legendre{D_1}{-1}|D_1|$ and $G(\chi_D)^2=\legendre{D}{-1}|D|$ in the last equality (see e.g \cite[Corollary 2.1.47 on p.~33]{Cohenbook}). On the other hand, we have by our equivalent definition \eqref{eq:defofUslash} of the $U$ operator that
    \begin{align}
        &U_{|D_2|}([G_{k,D_1,D_2}(z),G_{k,D_1,D_2}(z)]_{2e})
        \\=&\sum_{v~\text{mod}~|D_2|}|D_2|^{(2k+4e)/2-1}[G_{k,D_1,D_2}(z),G_{k,D_1,D_2}(z)]_{2e}\bigg|_{2k+4e}\begin{bmatrix}
            1 & v\\ 0 & |D_2|
        \end{bmatrix}\\=&\sum_{v~\text{mod}~|D_2|}|D_2|^{(2k+4e)/2-1}\left[G_{k,D_1,D_2}(z)\bigg|_k\begin{bmatrix}
            1 & v\\ 0 & |D_2|
        \end{bmatrix},G_{k,D_1,D_2}(z)\bigg|_k\begin{bmatrix}
            1 & v\\ 0 & |D_2|
\end{bmatrix}\right]_{2e}\\=&\sum_{v~\text{mod}~|D_2|}|D_2|^{(2k+4e)/2-1}\left[|D_2|^{-k/2}G_{k,D_1,D_2}\left(\frac{z+v}{|D_2|}\right),|D_2|^{-k/2}G_{k,D_1,D_2}\left(\frac{z+v}{|D_2|}\right)\right]_{2e}\\=&\sum_{v~\text{mod}~|D_2|}|D_2|^{2e-1}\left[G_{k,D_1,D_2}\left(\frac{z+v}{|D_2|}\right),G_{k,D_1,D_2}\left(\frac{z+v}{|D_2|}\right)\right]_{2e}.
    \end{align}
    It follows that 
    \begin{align}
       &\quad \,\sum_{D=D_1D_2}\left(\frac{D_2}{-1}\right)|D_{2}|^{-2e}U_{|D_2|}([G_{k,D_1,D_2}(z),G_{k,D_1,D_2}(z)]_{2e})\\&=\sum_{D_1D_2=D}\sum_{v~\text{mod}~|D_2|}\left(\frac{D_2}{-1}\right)|D_2|^{-1}\left[G_{k,D_1,D_2}\left(\frac{z+v}{|D_2|}\right),G_{k,D_1,D_2}\left(\frac{z+v}{|D_2|}\right)\right]_{2e}\\&=\Tr_1^D([G_{k,D}(z),G_{k,D}(z)]_{2e}),
    \end{align}
    as desired.
\end{proof}
\begin{proposition}\label{prop:FourierexofGDe}
    Let $k\geq4$ and $e>0$ and let $D$ be an odd fundamental discriminant with $(-1)^kD>0$. Then 
    \begin{align}
      \mathcal{G}_{D,k,e}(z)=\sum_{D=D_1D_2}\legendre{D_2}{-|D_1|}|D_2|^{-e}U_{|D_2|}([G_{k,D_1,D_2}(4z),\theta(|D_1|z)]_e)
    \end{align}
    where the summation is over all decompositions of $D$ as a product of two fundamental discriminants, and $U_{|D_2|}$ is the map defined in \eqref{eq:defofUmap}.
\end{proposition}
\begin{proof}

The proof follows a similar outline to Proposition \ref{prop:FourierexFDE}. From Proposition \ref{prop:projectioncomputation}, we know that $\mathcal{G}_{D, k,e}(z)=\Tr^{4D}_4\left[G_{k,D}(4z),\theta(|D|z)\right]_e$.
We use the coset representatives (Lemma \ref{lem:cosetreps}) for $\Gamma_0(4|D|)\backslash \Gamma_0(4)$,
\begin{align}
\left\{\gamma_{D_1,\mu}=\begin{bmatrix}
                1 & 0 \\ 4|D_1| & 1
\end{bmatrix}\begin{bmatrix}
            1 & \mu \\ 0 & 1 
\end{bmatrix}: \qquad \text{where } D=D_1D_2,\quad \mu \pmod {|D_2|}\right\},\label{eq:cosets4Dto4}
\end{align}
where $D=D_1D_2$ is a product of fundamental discriminants. By a simple casework we have
\begin{align}\frac{\varepsilon_{|D_1|}}{\varepsilon_{|D|}\cdot\varepsilon_{|D_2|}}\legendre{D_1}{|D_2|}=\legendre{D_2}{-|D_1|}.\label{eq:signofkronecker}\end{align}
Now Lemmas \ref{lem:FourierexpansionofGkD} and \ref{lem:Fourierexpansiontheta} together with \eqref{eq:signofkronecker} imply that
\begin{align*}
      &\quad \,\sum_{D_1D_2=D}\sum_{\mu~\text{mod}~|D_2|}[G_{k,D}(4z),\theta(|D|z)]_e\big|_{k+2e+\frac{1}{2}}\gamma_{D_1,\mu} \\
        &=\sum_{D_1D_2=D}\sum_{\mu~\text{mod}~|D_2|}\left[G_{k,D}(4z)|_k\gamma_{D_1,\mu},\theta(|D|z)\big|_{\frac{1}{2}}\gamma_{D_1,\mu}\right]_e \\
        &=\sum_{D_1D_2=D}\sum_{\mu~\text{mod}~|D_2|} \frac{\varepsilon_{|D_1|}}{\varepsilon_{|D|}\cdot
        \varepsilon_{|D_2|}}\legendre{D_1}{|D_2|}|D_2|^{-1}\\
        &\qquad  \times\left[G_{k,D_1,D_2}\left(\frac{4z+|D_1|^*+4\mu}{|D_2|}\right),\theta\left(\frac{|D_1|z+4^*+|D_1|\mu}{|D_2|}\right)\right]_e
        \\
        &=\sum_{D_1D_2=D}\sum_{\mu~\text{mod}~|D_2|}\legendre{D_2}{-|D_1|}|D_2|^{-1} \\
         &\qquad \times \left[G_{k,D_1,D_2}\left(\frac{4(z+4^{\ast}|D_1|^*+\mu)}{|D_2|}\right),\theta\left(\frac{|D_1|(z+4^*|D_1|^*+\mu)}{|D_2|}\right)\right]_e \\        &=\hspace{-13pt}\sum_{D_1D_2=D}\sum_{\nu~\text{mod}~|D_2|}\hspace{-5pt}\legendre{D_2}{-|D_1|}\hspace{-3pt}|D_2|^{-1+k/2+1/4}\hspace{-2pt}\left[G_{k,D_1,D_2}(4z)\bigg|_k\begin{bmatrix}
            1 & \nu \\ 0 & |D_2|
        \end{bmatrix},\theta(|D_1|z)\bigg|_{\frac{1}{2}}\begin{bmatrix}
            1 & \nu \\ 0 & |D_2|
        \end{bmatrix}\hspace{-2pt}\right]_e \\
&=\hspace{-5pt}\sum_{D_1D_2=D}\sum_{\nu~\text{mod}~|D_2|}\legendre{D_2}{-|D_1|}|D_2|^{-1+k/2+1/4}\left[G_{k,D_1,D_2}(4z),\theta(|D_1|z)\right]_e\bigg|_{2e+k+\frac{1}{2}}\begin{bmatrix}
            1 & \nu \\ 0 & |D_2|
        \end{bmatrix} \\
        &=\sum_{D_1D_2=D} \legendre{D_2}{-|D_1|}|D_2|^{-e}U_{D_2}[G_{k,D_1,D_2}(4z),\theta(|D_1|z)]_e,
    \end{align*}
    as desired.
\end{proof}

We are now ready to compute the Fourier expansions of $\mathcal{F}_{D,k,e}$ and $\mathcal{S}_D(\mathcal{G}_{D,k,e})$.
\begin{lemma} \label{lem:ffouriercoefficient}
    Let $k\ge4$, $e>0$ and let $D$ be an odd fundamental discriminant with $(-1)^{k}D>0$. Then we have the Fourier expansion
    \begin{align}
        \mathcal{F}_{D,k,e}(z)= \sum_{n\ge 1}f_{D,k,e}(n)q^n,
    \end{align}
    where
    \begin{align}
        f_{D, k, e}(n)&= \sum_{D=D_1D_2}\legendre{D_2}{-1}|D_2|^{-2e}\sum_{\substack{a_1,a_2\ge 0\\ a_1+a_2=n|D_2|}}\sum_{d|(a_1, a_2)}\legendre{D}{d}d^{k-1}\sigma_{k-1, D_1, D_2}\left(\frac{a_1a_2}{d^2}\right) C_{e, a_1, a_2},\\ \label{eq:ffouriercoefficient}
        C_{e,a_1,a_2}&=\sum_{r=0}^{2e}(-1)^ra_1^ra_2^{2e-r}\binom{2e+k-1}{2e-r}\binom{2e+k-1}{r}.
    \end{align}
\end{lemma}

\begin{proof}
By Proposition \ref{prop:FourierexFDE}, we have
\begin{align}
    f_{D,k,e}(n)=\sum_{D=D_1D_2}\legendre{D_2}{-1}|D_{2}|^{-2e}F_{D_1, D_2, e}(n),
\end{align}
where $F_{D_1, D_2, e}(n)$ is the $n|D_2|$-th Fourier coefficient of 
$[G_{k,D_1,D_2}(z),G_{k,D_1,D_2}(z)]_{2e}.$
Note that
\begin{align}
    G_{k,D_1,D_2}(z)^{(r)} = \sum_{n\ge 0} n^r \sigma_{k-1, D_1, D_2}(n)q^n,
\end{align}
which implies that the $n|D_2|$-th Fourier coefficient of $G_{k,D_1,D_2}^{(r)}(z)G_{k,D_1,D_2}^{(2e-r)}(z)$ is
\begin{align}
    \sum_{\substack{a_1, a_2\geq 0\\a_1+a_2=n|D_2|}} a_1^r \sigma_{k-1, D_1, D_2}(a_1)a_2^{2e-r} \sigma_{k-1, D_1, D_2}(a_2).
\end{align}
It follows that $F_{D_1, D_2, e}(n)=$
\begin{align}
     &\hspace{30pt} \sum_{r=0}^{2e}(-1)^r\binom{2e+k-1}{2e-r}\binom{2e+k-1}{r} \sum_{\substack{a_1, a_2\geq 0\\a_1+a_2=n|D_2|}} a_1^r \sigma_{k-1, D_1, D_2}(a_1)a_2^{2e-r} \sigma_{k-1, D_1, D_2}(a_2)\\
     &=\sum_{\substack{a_1, a_2\geq 0\\a_1+a_2=n|D_2|}} \sigma_{k-1, D_1, D_2}(a_1)\sigma_{k-1, D_1, D_2}(a_2)\sum_{r=0}^{2e}a_1^ra_2^{2e-r}(-1)^r\binom{2e+k-1}{2e-r}\binom{2e+k-1}{r}\\
     &=\sum_{\substack{a_1, a_2\geq 0\\a_1+a_2=n|D_2|}} \sigma_{k-1, D_1, D_2}(a_1)\sigma_{k-1, D_1, D_2}(a_2)C_{e, a_1, a_2}\\
     &=\sum_{\substack{a_1,a_2\ge 0\\ a_1+a_2=n|D_2|}}\sum_{d|(a_1, a_2)}\left(\frac{D}{d}\right)d^{k-1}\sigma_{k-1, D_1, D_2}\left(\frac{a_1a_2}{d^2}\right) C_{e, a_1, a_2}.
\end{align}
where the last equality is given by the Hecke multiplicative relation \cite[p. 194]{Kohnen-Zagier1981}
\begin{align}
    \sigma_{k-1, D_1, D_2}(a_1)\sigma_{k-1, D_1, D_2}(a_2)
    &=\sum_{d|(a_1,a_2)}\left(\frac{D}{d}\right)d^{k-1}\sigma_{k-1, D_1, D_2}\left(\frac{a_1a_2}{d^2}\right).
\end{align}
This finishes the proof.
\end{proof}
\begin{lemma} \label{lem:gfouriercoefficients}
Let $k\geq4$, $e>0$ and let $D$ be an odd fundamental discriminant with $(-1)^k D>0$. Then we have the Fourier expansion
\begin{align}\mathcal{S}_D\left(\mathcal{G}_{D,k,e}(z))\right) = \sum_{n \geq 1} g_{D,k,e}(n)q^n,\end{align}
where
\begin{equation}
g_{D,k,e}(n)= |D|^{e}\sum\limits_{D=D_{1}D_{2}}\left(\frac{D_{2}}{-1}\right)|D_{2}|^{-2e}\sum_{\substack{a_{1},a_{2}\geq0\\
a_{1}+a_{2}=n|D_{2}|
}
}\sum_{d|(a_{1},a_{2})}\left(\frac{D}{d}\right)d^{k-1}\sig\left(\frac{a_{1}a_{2}}{d^{2}}\right)E(a_1, a_2), 
\end{equation}
\begin{equation}
E(a_1, a_2)=\sum\limits_{r=0}^{e}(-1)^{r}\binom{e+k-1}{e-r}\binom{e-1/2}{r}4^{r}\left(a_{1}a_{2}\right)^{r}(a_{2}-a_{1})^{2(e-r)}.\label{eq:gfouriercoefficients}
\end{equation}
\end{lemma}
\begin{proof}
From Proposition \ref{prop:FourierexofGDe} and the definition of the Rankin-Cohen bracket \eqref{eq:defofrankin-cohenbracket}, we have $\mathcal{G}_{D,k,e}(z)=$
\begin{align}
\sum_{D=D_1D_2}\legendre{D_2}{-|D_1|}|D_2|^{-e}U_{|D_2|}\left(\sum_{r=0}^e(-1)^r\binom{e+k-1}{e-r}\binom{e-1/2}{r}G_{k,D_1,D_2}(4z)^{(r)}\theta(|D_1|z)^{(e-r)}\right).
\end{align}
Hence $\mathcal{S}_D(\mathcal{G}_{D,k,e}(z))=$
\begin{align}
\sum\limits_{D=D_{1}D_{2}}\hspace{-2pt}\legendre{D_{2}}{-|D_{1}|}|D_{2}|^{-e}
\sum\limits_{r=0}^{e}(-1)^{r}\binom{e+k-1}{e-r}\binom{e-1/2}{r}\mathcal{S}_{D}\hspace{-2pt}\left(U_{|D_{2}|}(G_{k,D_{1},D_{2}}(4z)^{(r)}\theta(|D_{1}|z)^{(e-r)})\right)\hspace{-1pt},
\end{align}
where we abuse notation to move the Shimura operator $\mathcal{S}_D$ into the sums. Now taking the Fourier expansions of $\Gk^{(r)}$ and $\theta(|D_1|z)^{(e-r)}$, we get 

\[G_{k,D_{1},D_{2}}(4z)^{(r)}=\sum_{n\geq0}(4n)^{r}\sigma_{k-1,D_{1},D_{2}}(n)q^{4n}\]
\[\theta(|D_{1}|z)^{(e-r)}=\sum_{n\in\mathbb{Z}}(n^{2}|D_{1}|)^{e-r}q^{n^{2}|D_{1}|}\]
where $\sigma_{k-1, D_1, D_2}(n)$ is defined in \cite[p. 193]{Kohnen-Zagier1981}. This allows us to rewrite the product
\begin{align}
G_{k,D_{1},D_{2}}(4z)^{(r)}\theta(|D_{1}|z)^{(e-r)}&=\left(\sum_{s\geq0}(4s)^{r}\sigma_{k-1,D_{1},D_{2}}(s)q^{4s}\right)\left(\sum_{m\in\mathbb{Z}}(m^{2}|D_{1}|)^{e-r}q^{m^{2}|D_{1}|}\right)\\
&=\sum_{n\geq0}c_{r}(n)q^{n},
\end{align}
where
\[c_{r}(n)=\sum_{m\equiv n\mod 2}\left(n-m^{2}|D_{1}|\right)^{r}\sigma_{k-1,D_{1},D_{2}}\left(\frac{n-m^{2}|D_{1}|}{4}\right)(m^{2}|D_{1}|)^{e-r},\]
taking the convention that $\sigma_{k-1,D_{1},D_{2}}(x)=0$ if $x\notin\mathbb{\mathbb{Z}}$ or $x<0$.
It follows that 
\begin{align}\Ud\left(G_{k,D_{1},D_{2}}(4z)^{(r)}\theta(|D_{1}|z)^{(e-r)}\right)=\Ud\left(\sum_{n\geq1}c_{r}(n)q^{n}\right)=\sum_{n\geq1}c_{r}(n|D_2|)q^{n}.\label{eq:coeffaD_2}\end{align}
Now we compute the $D$-th Shimura lift of \eqref{eq:coeffaD_2}. If we write 
\begin{align}
\mathcal{S}_D\left(\sum_{n\geq1}c_{r}(n|D_2|)q^{n}\right)=\sum_{n\geq1}a_{r,D_2}(n)q^{n}\end{align}
for some $a_{r,D_2}(n)$, then by the definition of $\mathcal{S}_D$ \eqref{eq:defofshimuralift}, we have $a_{r,{D_2}}(n)=$
\begin{align}
\sum_{d|n}\legendre{D}{d}d^{k+2e-1}\sum_{m\in\mathbb{Z}}\left(|D_{2}||D|\frac{n^{2}}{d^{2}}-|D_{1}|m^{2}\right)^{r}
\left(m^{2}|D_{1}|\right)^{e-r}\sig\left(\frac{|D_{2}||D|\frac{n^{2}}{d^{2}}-m^{2}|D_{1}|}{4}\right).
\end{align}
Note that we can write
\begin{align}\frac{|D_{2}||D|\frac{n^{2}}{d^{2}}-m^{2}|D_{1}|}{4}=|D_{1}|a_1a_2,\quad{\rm where}~a_{1}=\frac{|D_{2}|\frac{n}{d}+m}{2}~{\rm and}~a_{2}=\frac{|D_{2}|\frac{n}{d}-m}{2}.\end{align}
It follows that $a_{r,D_2}(n)=$
\begin{align}
&\hspace{30pt}\sum_{d|n}\legendre{D}{d}d^{k+2e-1}\sum_{\substack{a_{1},a_{2}\geq0\\
a_{1}+a_{2}=\frac{n}{d}|D_{2}|
}
}\left(4|D_{1}|a_{1}a_{2}\right)^{r}(a_{2}-a_{1})^{2(e-r)}|D_{1}|^{e-r}\sig(|D_{1}|a_{1}a_{2})\\
 & =\sum_{\substack{a_{1},a_{2}\geq0\\
a_{1}+a_{2}=n|D_{2}|
}
}\sum_{d|(a_{1},a_{2})}\legendre{D}{d}d^{k+2e-1}\left(4|D_{1}|\frac{a_{1}a_{2}}{d^2}\right)^{r}\left(\frac{a_{2}-a_{1}}{d}\right)^{2(e-r)}|D_{1}|^{e-r}\sig\left(|D_{1}|\frac{a_{1}a_{2}}{d^{2}}\right)\\
&=\sum_{\substack{a_{1},a_{2}\geq0\\
a_{1}+a_{2}=n|D_{2}|
}
}\sum_{d|(a_{1},a_{2})}\legendre{D}{d}\legendre{D_{2}}{|D_{1}|}d^{k-1}|D_{1}|^{e}\left(4a_{1}a_{2}\right)^{r}(a_{2}-a_{1})^{2(e-r)}\sig\left(\frac{a_{1}a_{2}}{d^{2}}\right).
\label{eq:FouriercD2(n)}\end{align}
Now we substitute \eqref{eq:FouriercD2(n)} back into our equation for $\mathcal{S}_D(\mathcal{G}_{D,k,e}(z))$. Let $$\mathcal{S}_D(\mathcal{G}_{D,k,e}(z)) = \sum_{n\geq 1} g_{D,k,e}(n)q^n.$$ 
Then we have $g_{D,k,e}(n)=$
\begin{align}
&\sum_{D=D_1D_2}\legendre{D_2}{-|D_1|}|D_2|^{-e}\sum_{r=0}^e\binom{e+k-1}{e-r}\binom{e-1/2}{r}a_{r,D_2}(n)\\=&
\sum\limits_{D=D_{1}D_{2}}\legendre{D_{2}}{-|D_{1}|}|D_{2}|^{-e}\sum\limits_{r=0}^{e}(-1)^{r}\binom{e+k-1}{e-r}\binom{e-1/2}{r}\\
 &\qquad \times\sum_{\substack{a_{1},a_{2}\geq0\\a_{1}+a_{2}=n|D_{2}|}}\sum_{d|(a_{1},a_{2})}\legendre{D}{d}\legendre{D_{2}}{|D_{1}|}d^{k-1}|D_{1}|^{e}\left(4a_{1}a_{2}\right)^{r}(a_{2}-a_{1})^{2(e-r)}\sig\left(\frac{a_{1}a_{2}}{d^{2}}\right)\\
=&\sum\limits_{D=D_{1}D_{2}}\left(\frac{D_{2}}{-|D_{1}|}\right)\legendre{D_{2}}{|D_{1}|}|D_{2}|^{-e}|D_{1}|^{e}\sum\limits_{r=0}^{e}(-1)^{r}\binom{e+k-1}{e-r}\binom{e-1/2}{r}\\
 &\qquad \times\sum_{\substack{a_{1},a_{2}\geq0\\
a_{1}+a_{2}=n|D_{2}|
}
}\sum_{d|(a_{1},a_{2})}\legendre{D}{d}d^{k-1}\left(4a_{1}a_{2}\right)^{r}(a_{2}-a_{1})^{2(e-r)}\sig\left(\frac{a_{1}a_{2}}{d^{2}}\right)\\
=&|D|^{e}\sum\limits_{D=D_{1}D_{2}}\legendre{D_{2}}{-1}|D_{2}|^{-2e}\sum_{\substack{a_{1},a_{2}\geq0\\a_{1}+a_{2}=n|D_{2}|}}\sum_{d|(a_{1},a_{2})}\legendre{D}{d}d^{k-1}\sig\left(\frac{a_{1}a_{2}}{d^{2}}\right)
\\
 &\qquad \times\sum\limits_{r=0}^{e}(-1)^{r}\binom{e+k-1}{e-r}\binom{e-1/2}{r}4^{r}\left(a_{1}a_{2}\right)^{r}(a_{2}-a_{1})^{2(e-r)}\\
=&|D|^{e}\sum\limits_{D=D_{1}D_{2}}\legendre{D_{2}}{-1}|D_{2}|^{-2e}\left(\sum_{\substack{a_{1},a_{2}\geq0\\
a_{1}+a_{2}=n|D_{2}|
}
}\sum_{d|(a_{1},a_{2})}\legendre{D}{d}d^{k-1}\sig\left(\frac{a_{1}a_{2}}{d^{2}}\right)E(a_1, a_2)\right),
\end{align}
as desired.
\end{proof}
\section{Discussion}
\label{sec:discussion}
It is a folklore conjecture that $S_{2\ell}^{0,D}(1)=S_{2\ell}(1)$. Luo \cite{Luononvanishing2015} showed that for $\ell$ sufficiently large one has $\dim S_{2\ell}^{0,1}(1)\gg \ell$. Our Theorem \ref{thm:gspan} (the case $D=1$ was proved earlier by Xue \cite[Proposition 3.5]{Xue-Selbergidentity}) provides a possible different approach to the conjecture. By studying the linear independence of $\mathcal{G}_{D,k,e}$ or $\mathcal{F}_{D,k,e}$, one could obtain lower bounds on the dimension of $S_{2\ell}^{0,D}(1)$.

\begin{conjecture}
For $\ell$ even, $D$ a positive fundamental discriminant, the set $\{\mathcal{G}_{D, k,e}~|~k+2e=\ell, 1 \leq e \leq \lfloor \frac{\ell}{6}\rfloor\}$ is linearly independent.
\end{conjecture}
We checked this conjecture computationally in the $D=1$ case up to $\ell=1000$ and for prime $D$ less than 50 up to $\ell=100$, using code written in Pari/GP \cite{paricode}. In particular, we computationally verified that the matrix 
\[\begin{bmatrix} g_{D, \ell-2,1}(4) & g_{D, \ell -2,1}(8) & \ldots & g_{D, \ell -2,1}(4\lfloor \frac{\ell}{6}\rfloor) \\ 
 g_{D, \ell-4,2}(4) &  g_{D, \ell-4,2}(8) & \ldots & g_{D, \ell -4,2}(4\lfloor \frac{\ell}{6}\rfloor)\\
 \vdots & \vdots & \ddots & \vdots \\
   g_{D,\ell-2\lfloor \frac{\ell}{6}\rfloor, \lfloor \frac{\ell}{6}\rfloor}(4) &  g_{D,\ell-2\lfloor \frac{\ell}{6}\rfloor, \lfloor \frac{\ell}{6}\rfloor}(8) & \ldots & g_{D,\ell -2\lfloor \frac{\ell}{6}\rfloor,\lfloor \frac{\ell}{6}\rfloor}(4\lfloor \frac{\ell}{6}\rfloor)\\
\end{bmatrix},\]
where $\mathcal{G}_{D,k,e} = \sum_{n \geq 1} g_{D,k,e}(n)q^n$, has nonzero determinant. Further work in this area should try to prove that this determinant is nonzero in general. 

The conjecture would have several interesting consequences. Using the isomorphism between $S^{0,D}_{\ell+1/2}(4)$ and $S^{0,D}_{2\ell}(1)$ given by the $D$-th Shimura lift, we find that the dimension of $S^{0,D}_{2\ell}$ would be at least $\lfloor \frac{\ell}{6}\rfloor$. Since the dimension of $S_{2\ell}(1)$ for even $\ell$ is $\lfloor \frac{2\ell}{12} \rfloor = \lfloor \frac{\ell}{6} \rfloor$ and $S^{0,D}_{2\ell}(1) \subseteq S_{2\ell}(1)$, we would conclude that $S^{0,D}_{2\ell}(1)=S_{2\ell}(1)$, settling the  conjecture on the non-vanishing of twisted central $L$-values for Hecke eigenforms. 

This would then imply that $S^{0,D}_{\ell+1/2}(4) = S^{+}_{\ell+1/2}(4)$, so the Kohnen plus space for $k$ even is generated by Hecke eigenforms whose $D$-th coefficients are nonzero for all fundamental discriminants $D$. Further, we would conclude that $\{\mathcal{G}_{D, k,e}\}_{k+2e=\ell,1\leq e \leq \lfloor \frac{\ell}{6} \rfloor}$ is a basis for $S^{+}_{\ell+1/2}(4)$, and the set $\{\mathcal{G}_{D, k,e}\}_{k+2e=\ell, 0\leq e \leq \lfloor \frac{\ell}{6} \rfloor}$ is a basis for $M^{+}_{\ell+1/2}(4)$ (since the $0$-th Rankin-Cohen bracket produces a modular form which is non-cuspidal but still in the Kohnen plus space). To the best of our knowledge, a similar basis was first mentioned by Henri Cohen in a MathOverflow post.

\section*{Acknowledgements}
This research was supported by NSA MSP grant H98230-24-1-0033.

\bibliographystyle{plain}

\providecommand{\bysame}{\leavevmode\hbox
to3em{\hrulefill}\thinspace}

\bibliography{ref}

\vspace{.2in}

\end{document}